\documentclass[a4paper,10pt,reqno]{amsart}

\usepackage[english]{babel}
\usepackage[utf8]{inputenc}

\usepackage{amssymb}
\usepackage{amsmath}
\usepackage{amsthm}
\usepackage{bbm}

\usepackage{enumerate}
\usepackage{enumitem}
\usepackage{stmaryrd}
\usepackage{tikz}
\usepackage{tikz-cd}
\usepackage{marginnote}

\newcommand{\bbC}{\mathbb{C}}

\newcommand{\bbN}{\mathbb{N}}

\newcommand{\bbR}{\mathbb{R}}

\newcommand{\bbT}{\mathbb{T}}

\newcommand{\bbZ}{\mathbb{Z}}

\newcommand{\calB}{\mathcal{B}}

\newcommand{\calL}{\mathcal{L}}

\DeclareMathOperator{\id}{id} 
\DeclareMathOperator{\one}{\mathbbm{1}} 
\DeclareMathOperator{\re}{Re} 
\DeclareMathOperator{\im}{Im} 
\newcommand{\argument}{\mathord{\,\cdot\,}} 
\newcommand{\dx}{\;\mathrm{d}} 
\DeclareMathOperator{\fix}{fix} 
\newcommand{\norm}[1]{\left\lVert #1 \right\rVert} 
\newcommand{\modulus}[1]{\left\lvert #1 \right\rvert} 
\newcommand{\restricted}[1]{|_{#1}}
\DeclareMathOperator{\supp}{supp} 
\DeclareMathOperator{\opensupp}{op\,supp} 
\newcommand{\Cont}{\mathcal{C}} 

\newcommand{\spec}{\sigma} 
\newcommand{\per}{{\operatorname{per}}}
\newcommand{\pnt}{{\operatorname{pnt}}}
\newcommand{\specPnt}{\spec_{\pnt}} 
\newcommand{\spr}{r} 

\newcommand{\impliesProof}[2]{``\ref{#1} $\Rightarrow$ \ref{#2}''}
\newcommand{\equivalentProof}[2]{``\ref{#1} $\Leftrightarrow$ \ref{#2}''}
\newcommand{\rightProof}{``$\Rightarrow$''}
\newcommand{\leftProof}{``$\Leftarrow$''}

% ...

\theoremstyle{definition}
\newtheorem{definition}{Definition}[section]
\newtheorem{remark}[definition]{Remark}

\newtheorem{example}[definition]{Example}

\theoremstyle{plain}

\newtheorem{lemma}[definition]{Lemma}
\newtheorem{theorem}[definition]{Theorem}
\newtheorem{corollary}[definition]{Corollary}

\numberwithin{equation}{section}

\begin{document}

\title[Aperiodicity of positive operators]{Aperiodicity of positive operators that increase the support of functions}
\author{Jochen Gl\"uck}
\address{Bergische Universit\"at Wuppertal, Fakult\"at f\"ur Mathematik und Naturwissenschaften, Gaußstr.\ 20, 42119 Wuppertal, Germany}
\email{glueck@uni-wuppertal.de}
\date{\today}
\begin{abstract}
	Consider a positive operator $T$ on an $L^p$-space 
	(or, more generally, a Banach lattice)
	which increases the support of functions in the sense that 
	$\supp(Tf) \supseteq \supp{f}$ for every function $f \ge 0$.
	We show that this implies, under mild assumptions, 
	that $T$ has no unimodular eigenvalues except for possibly the number $1$.
	This rules out periodic behaviour of any orbits of the powers of $T$,
	and thus enables us to prove convergence of those powers in many situations.
	
	For the proof we first perform a careful analysis 
	of the action of lattice homomorphisms on the support of functions;
	then we split $T$ into an invertible and a weakly stable part, 
	and apply the aforementioned analysis to the invertible part.
	An appropriate adaptation of this argument allows us to prove another version of our main result 
	which is useful for the study of so-called irreducible operators.
\end{abstract}

\maketitle

\section{Introduction}

\subsection*{Motivation}

If a matrix $T \in \bbR^{d \times d}$ with spectral radius $\spr(T) = 1$ has only entries $\ge 0$ and all diagonal entries of $T$ are non-zero, then it follows from classical Perron--Frobenius theory that $1$ is the only eigenvalue of $T$ in the complex unit circle $\bbT$.
The latter spectral property if often referred to as \emph{aperiodicity}. 

For linear operators $T$ that leaves the positive cone of an infinite dimensional function space (or Banach lattice) $E$ invariant, the most obvious generalization of the assumption that all diagonal entries be non-zero is the assumption that $T \ge \varepsilon \id_E$ for some $\varepsilon > 0$. 
While this does indeed imply a similar result for the spectrum of $T$ 
(see Section~\ref{section:domination-of-id} for details), this observation is only of minor usability since the assumption $T \ge \varepsilon \id_E$ is extremely strong in infinite dimensions.

\subsection*{Operators that increase the support of functions}

It is therefore worthwhile to look for a more prevalent properties which generalize the assumption that all diagonal entries of a matrix be non-zero.  
One such property on, say, a function space $L^p$ over a $\sigma$-finite measure space is as follows: 
one assumes that for every $0 \le f \in L^p$ the support $\supp(Tf)$ contains the support $\supp f$ up to a set of measure $0$. 
Here, the support of a function $f \in L^p(\Omega,\mu)$ means the set $\{\omega \in \Omega: \; f(\omega) \not= 0\}$;
this set is defined only up to a null set (and thus, strictly speaking, it is not a set but an element of the \emph{measure algebra} of $(\Omega,\mu)$).
Under mild assumptions this property implies that the \emph{point spectrum} $\specPnt(T)$ (rather than the spectrum) intersects the unit circle $\bbT$ at most in the point $1$:

\begin{theorem}
	\label{thm:main-everywhere}
	Let $(\Omega,\mu)$ be a $\sigma$-finite measure space and $p \in [1,\infty)$. 
	Let $T: L^p \to L^p$ be a positive linear operator on the complex-valued space $L^p := L^p(\Omega,\mu)$. 
	Assume that $T$ is power bounded and that, for each $0 \le f \in L^p$, 
	there exists an integer $n \ge 1$ such that $\supp(T^nf) \supseteq \supp(T^{n-1}f)$. 
	Then $\specPnt(T) \cap \bbT \subseteq \{1\}$.
\end{theorem}

As indicated above, the condition $\supp(T^nf) \supseteq \supp(T^{n-1}f)$ is meant as an inclusion up to a null set
(i.e., as in inclusion within the measure algebra).
\emph{Power boundedness} of $T$ means that $\sup_{n \in \bbN_0} \norm{T^n} < \infty$, 
\emph{positivity} of $T$ means that $Tf \ge 0$ whenever $f \ge 0$, and the inequalities $\le$ and $\ge$ for functions are to be understood almost everywhere.
The simplest case of the inclusion $\supp(T^nf) \supseteq \supp(T^{n-1}f)$ is of course the case $\supp(Tf) \supseteq \sup(f)$ for each $0 \le f \in L^p$; 
this latter property is, as explained before, an infinite-dimensional generalization of the property of a matrix $T \ge 0$ that all diagonal entries be non-zero, and it is much weaker than the assumption that $T \ge \varepsilon \id$ for some $\varepsilon > 0$.

Similarly as for matrices, one may think of the conclusion $\specPnt(T) \cap \bbT \subseteq \{1\}$ 
in Theorem~\ref{thm:main-everywhere} as an \emph{aperiodicity} property,
since it is equivalent to the assertion that the powers of $T$ do not have a periodic orbit with minimal period $\ge 2$.
This explains the title of the paper, and we will employ this property to prove 
various convergence theorems for the powers of $T^n$ as $n \to \infty$. 
We remark that, at least when the spectral radius of $T$ satisfies $\spr(T) = 1$, 
the property $\specPnt(T) \cap \bbT \subseteq \{1\}$ is, in some parts of the literature, 
also called \emph{primitivity}; 
see for instance \cite[Section~6]{Grobler1995}.

We will actually prove a more general version of Theorem~\ref{thm:main-everywhere} which also holds on Banach lattices 
(see Subsection~\ref{subsection:primitivity-1:point-spectrum}; there we will also show how to derive Theorem~\ref{thm:main-everywhere} from its Banach lattice version). 
The Banach lattice setting does not only make the result available for other functions spaces, too; 
it is also a more natural setting for the proof which will make heavy usage of lattice theory and will, in particular, include a reduction to a lattice subspace which will not be an $L^p$-space, in general, even if we start on an $L^p$-space.

\subsection*{Irreducible operators and partially increasing support}

In some cases it will happen that the property $\supp(Tf) \supseteq \sup(f)$ is not satisfied for all $0 \le f \in L^p$, but only for those $f$ which are supported within a given set $S$ that is smaller than the underlying measure space. 
Still, we are able to essentially make the same conclusion if we add the assumption that the operator $T$ be \emph{irreducible}, which means that, for each non-zero $0 \le f \in L^p$ and each non-zero linear functional $g \ge 0$ on $L^p$, there exists an integer $n \ge 0$ such that $\langle g, T^n f \rangle > 0$.

\begin{theorem}
	\label{thm:main-irred}
	Let $(\Omega,\mu)$ be a $\sigma$-finite measure space and $p \in [1,\infty)$. 
	Let $T: L^p \to L^p$ be a positive linear operator on the complex-valued space $L^p = L^p(\Omega,\mu)$. 
	Assume that $T$ is power bounded and irreducible and that there exists a measurable set $S \subseteq \Omega$ 
	of non-zero measure such that $\supp(Tf) \supseteq \supp(f)$ for each $0 \le f \in L^p$ 
	that satisfies $\supp (f) \subseteq S$. 
	Then $\specPnt(T) \cap \bbT \subseteq \{1\}$.
\end{theorem}

In contrast to Theorem~\ref{thm:main-everywhere} we do not allow for powers of $T$ 
in the inclusion $\supp(Tf) \supseteq \supp(f)$ here. 
We do not know what happens if we only assume $\supp(T^nf) \supseteq \supp(T^{n-1}f)$ 
in the situation of Theorem~\ref{thm:main-irred}.
Again, we will show Theorem~\ref{thm:main-irred} in the more general setting of Banach lattices; see Subsection~\ref{subsection:conditions-for-primitivity-2:point-spectrum}, where we also prove Theorem~\ref{thm:main-irred} as a consequence of its Banach lattice version.

As is the case for Theorem~\ref{thm:main-everywhere}, Theorem~\ref{thm:main-irred} is also a generalization of a well-known result about matrices: if $0 \le T \in \bbR^{d \times d}$ has spectral radisu $\spr(T) = 1$ and is irreducible, and at least one diagonal entry of $T$ is non-zero, $1$ is the only eigenvalue of $T$ in the complex unit circle; see for instance \cite[Corollary~1 to Theorem~I.6.5 on pp.\,22--23]{Schaefer1974}. 
Another infinite-dimensional generalization of this finite-dimensional result can be found in \cite[Theorem~6.2]{Grobler1995} and is applicable to operators that belong to an operator ideal which admits a continuous trace.
However, such operators always have a compact power, which severaly restricts the applicability of this result; in contrast, Theorem~\ref{thm:main-irred} does not entail any such compactness condition.

\subsection*{Organization of the article}

In Sections~\ref{section:conditions-for-primitivity-1} and~\ref{section:conditions-for-primitivity-2} we prove 
various versions of Theorems~\ref{thm:main-everywhere} and~\ref{thm:main-irred}, respectively.
In Section~\ref{section:domination-of-id} we study how the assumption $T \ge \varepsilon \id$, and versions thereof, impact the spectrum of $T$. 
All three sections contain a number of applications of our spectral theoretic results to the analysis of the long-term behaviour of the powers of $T$.

\subsection*{Related literature}

Results that use different kind of support expansion properties for positive operators or operator semigroups
in order to obtain spectral or asymptotic results 
occur in the literature on various occassions.

Apart from finite dimensional theorems such as \cite[Corollary~1 to Theorem~I.6.5 on pp.\,22--23]{Schaefer1974} 
we mention results that rely on different kinds of \emph{positivity improving properties}
(e.g.\ \cite[Theorem~4.3]{Gerlach2013}, \cite{GlueckWeber2020} \cite[Theorem~6.1]{Grobler1995}, 
\cite[Theorem~2.2]{KellerLenzVogtWojciechowski2015}, \cite{MakarowWeber2000}, \cite[Corollary~2 on p.\,249]{Rudnicki1995}), 
support overlapping operators (see e.g.\ \cite{BartoszekBrown1998}, \cite{Lin2000} \cite[Corollary~1 on p.\,248]{Rudnicki1995}), 
operators that satisfy lower bound properties (such as e.g.\ in 
\cite{Ding2003, Lasota1982, Emelyanov2004a, Zalewska-Mitura1994} and \cite[Section~4]{GlueckWolff2019}) 
and, more losely related, $C_0$-semigroups on sequences spaces 
\cite{Davies2005, Keicher2006, Wolff2008} 
(see in particular the arguments used in the proofs of 
\cite[Proposition~3.5]{Keicher2006} and \cite[Proposition~2.3]{Wolff2008}).

\subsection*{Principal ideals and quasi-interior points in Banach lattices} 

Throughout the paper we assume the reader to be familiar with the basic theory of real and complex Banach lattices; 
standard references for this theory are e.g.\ 
the monographs \cite{AliprantisBurkinshaw2006, Meyer-Nieberg1991, Schaefer1974, Zaanen1983}. 

Let $E$ be a real or complex Banach lattice; we denote its cone by $E_+$. 
For every vector $u \in E_+$ we denote the smallest ideal in $E$ that contains $u$ by $E_u$; 
it is called the \emph{principal ideal generated by $u$}, and one can easily check that it is given by the formula
\begin{align*}
	E_u = \{f \in E: \, \exists c \ge 0 \; \modulus{f} \le cu\}.
\end{align*}
The vector $u \in E_+$ is called a \emph{quasi-interior point} (of $E_+$) if the principal ideal $E_u$ is dense in $E$.
If $E$ is an $L^p$-space over a $\sigma$-finite measure space for some $p \in [1,\infty)$, then a function $0 \le u \in L^p$ is a quasi-interior point if and only if it is strictly positive almost everywhere.

For elements $v,w \ge 0$ in a Banach lattice $E$ we will often consider the property
\begin{align*}
	\overline{E_w} \supseteq \overline{E_v},
\end{align*}
where the bar denotes the closure. 
Note that the set $\overline{E_w}$ is the smallest closed ideal in $E$ that contains $w$.
In the $L^p$-setting (over a $\sigma$-finite measure space and for $p \in [1,\infty)$), it is not difficult to check 
that the inclusion $\overline{E_w} \supseteq \overline{E_v}$ 
is equivalent to the inclusion $\supp(w) \supseteq \supp(v)$, 
which is again understood to hold up to a set of measure $0$.
Hence, the inclusion between the closed ideals generated by $w$ and $v$ 
will serve as an abstract Banach lattice theoretic version of the assumption that the support of $w$ contains the support of $v$.

However, it is important to point out that things become more complicted
if we consider spaces of continuous functions rather than $L^p$-spaces; 
see the discussion before Example~\ref{ex:weakly-expanding}, 
as well as Theorem~\ref{thm:everywhere-c-k} and Lemma~\ref{lem:ideal-inclusion-c-k}.

\subsection*{Notational conventions}

We use the conventions $\bbN := \{1,2,3, \dots\}$ and $\bbN_0 := \bbN \cup \{0\}$, 
as well as $\bbT := \{\lambda \in \bbC: \, \modulus{\lambda} = 1\}$.

\section{Operators that increase the support of functions} 
\label{section:conditions-for-primitivity-1}

\subsection{The point spectrum}
\label{subsection:primitivity-1:point-spectrum}

The main result of this subsection, and also of the entire Section~\ref{section:conditions-for-primitivity-1} 
is the following theorem, which is an abstract version of Theorem~\ref{thm:main-everywhere}.

\begin{theorem} 
	\label{thm:main-everywhere-bl}
	Let $T: E \to E$ be a positive and weakly almost periodic linear operator on a complex Banach lattice $E$.
	Assume that for each $f \in E_+$ there exists an integer $n \ge 1$ 
	such that $\overline{E_{T^nf}} \supseteq \overline{E_{T^{n-1}f}}$. 
	Then $\specPnt(T) \cap \bbT \subseteq \{1\}$.
\end{theorem}

The property \emph{weak almost periodicity} means that for each $f \in E$ 
the orbit $\{T^nf: \, n \ge 0\}$ is relatively compact with respect to the weak topology on $E$. 
By the uniform boundedness principle this property is stronger than being power bounded. 

Before we proceed to the proof of Theorem~\ref{thm:main-everywhere-bl} 
we show in the following corollary 
that the weak almost periodicity assumption can be replaced with the weaker property power boundedness
if the Banach lattice is a so-called \emph{KB-space}. 
A KB-space is a Banach lattice $E$ in which every increasing norm bounded sequence in the positive cone converges in norm. 
For instance, every $L^p$-space for $p \in [1,\infty)$ is a KB-space.

\begin{corollary} 
	\label{cor:main-everywhere-bl-kb}
	Let $T: E \to E$ be a positive and power-bounded linear operator on a KB-space $E$.
	Assume that for each $f \in E_+$ there exists an integer $n \ge 1$ 
	such that $\overline{E_{T^nf}} \supseteq \overline{E_{T^{n-1}f}}$. 
	Then $\specPnt(T) \cap \bbT \subseteq \{1\}$.
\end{corollary}

\begin{proof}
	Let $\lambda \in \bbT$ be an eigenvalue of $T$ and let $f \in E$ be an associated eigenvector; 
	we need to show that $\lambda=1$.
	
	Observe that $\modulus{f} = \modulus{Tf} \le T \modulus{f}$. 
	By iterating this inequality we see that the sequence $(T^n \modulus{f})$ is increasing.
	Due to the power boundedness of $T$, the sequence is also norm bounded, 
	so it converges to a point $h \in E_+$ as $E$ is a KB-space. 
	Obviously, $h$ is a fixed point of $T$. 
	
	We now consider the subspace $F := \overline{E_h}$ of $E$,
	which is a Banach lattice in its own right.
	The space $F$ is invariant under $T$ and contains the eigenvector $f$. 
	The restriction $T\restricted{F}: F \to F$ still satisfies the assumptions of the corollary.
	Moreover, the restriction is a weakly almost periodic operator on $F$.
	Indeed, the \emph{order interval}
	\begin{align*}
		[0,h] := \{f \in E: \, 0 \le f \le g\}
	\end{align*}
	is weakly compact%
	\footnote{
		Indeed, as is for instance explained 
		after Definition~2.4.11 on p.\,92 of \cite{Meyer-Nieberg1991}, 
		every KB-space has \emph{order continuous norm} 
		(see our discussion before Corollary~\ref{cor:irred-bl-oc} for a definition of this notion), 
		and having order continuous norm is equivalent to all order intervals being weakly compact
		(see \cite[Theorem~2.4.2 on p.\,86]{Meyer-Nieberg1991} or \cite[Theorem~II.5.10 on p.\,89]{Schaefer1974}).
	}
	and invariant under $T$, so the orbit of every $f \in [0,h]$ under $T$ is relatively weakly compact. 
	Since the span of $[0,h]$ is dense in $F$ and $T$ is power bounded, 
	it thus follows that $T\restricted{F}$ is indeed weakly almost periodic 
	\cite[Corollary~A.5(c) on p.\,514]{Engel2000}.
	
	So $T\restricted{F}$ satisfies all assumptions of Theorem~\ref{thm:main-everywhere-bl}. 
	Hence, the theorem implies that $\lambda = 1$.
\end{proof}

We do not know whether the assertion of Theorem~\ref{thm:main-everywhere-bl} 
remains true on general Banach lattices if $T$ is only power bounded rather than weakly almost periodic.
Next, observe that the corollary immediately gives us the first theorem from the introduction 
as a special case:

\begin{proof}[Proof of Theorem~\ref{thm:main-everywhere}]
	Since $p \in [1,\infty)$, the Banach lattice $E := L^p$ is a KB-space.
	Moreover, as explained at the end of the introduction, 
	the assumption on the supports in Theorem~\ref{thm:main-everywhere}
	is the same as the inclusion $\overline{E_{T^nf}} \supseteq \overline{E_{T^{n-1}f}}$.
	So all assumptions of Corollary~\ref{cor:main-everywhere-bl-kb} are satisfied.
\end{proof}

Let us now turn to the proof of Theorem~\ref{thm:main-everywhere-bl}. 
We first show a few properties of closed ideals generated by single vectors in the following three lemmas.

\begin{lemma}
	\label{lem:ideal-inclusion}
	Let $E$ be a real or complex Banach lattice and let $f,g \in E_+$. 
	The following assertions are equivalent:
	\begin{enumerate}[label=\upshape(\alph*)]
		\item\label{lem:ideal-inclusion:item:inclusion} 
		$\overline{E_f} \subseteq \overline{E_g}$.
		
		\item\label{lem:ideal-inclusion:item:element}
		$f \in \overline{E_g}$.
		
		\item\label{lem:ideal-inclusion:item:approx} 
		$\lim_{t \to \infty} f \land (tg) = f$ 
		(where the limit to be understood in norm).
		
		\item\label{lem:ideal-inclusion:item:small} 
		$\norm{(g-sf)^-} = o(s)$ as $s \downarrow 0$.
	\end{enumerate}
\end{lemma}

\begin{proof}
	\impliesProof{lem:ideal-inclusion:item:inclusion}{lem:ideal-inclusion:item:element}
	This implication is obvious.
	
	\impliesProof{lem:ideal-inclusion:item:element}{lem:ideal-inclusion:item:inclusion}
	This holds since $\overline{E_g}$ is a closed ideal 
	and $\overline{E_f}$ is the smalles closed ideal that contains $f$.
	
	\impliesProof{lem:ideal-inclusion:item:element}{lem:ideal-inclusion:item:approx}
	Due to~\ref{lem:ideal-inclusion:item:element} there exists a sequence of number $(c_n)$ in $[0,\infty)$ 
	and a sequence $(h_n)$ in $E$ which converges to $f$ 
	and which satisfies $\modulus{h_n} \le c_n g$ for each index $n$.
	By replacing each vector $h_n$ with the positive part of its real part 
	we may, and shall, assume that $0 \le h_n \le c_n g$ for each $n$. 
	
	Let $\varepsilon > 0$. 
	As $(f \land h_n)$ converges to $f$, 
	there exists an index $n_0$ such that $\norm{f - f \land h_{n_0}} \le \varepsilon$.
	For $t \ge c_{n_0}$ we have
	\begin{align*}
		0 
		\le 
		f \land h_{n_0} 
		\le 
		f \land (c_{n_0} g) \
		\le 
		f \land (t g) \le f,
	\end{align*}
	so $\norm{f - f \land (t g)} \le \varepsilon$.
	
	\impliesProof{lem:ideal-inclusion:item:approx}{lem:ideal-inclusion:item:element}
	If~\ref{lem:ideal-inclusion:item:approx} holds, then 
	$\big( f \land (ng) \big)$ is a sequence in $E_g$ that converges to $f$.
	
	\equivalentProof{lem:ideal-inclusion:item:element}{lem:ideal-inclusion:item:small}
	 A proof of this equivalence can, for instance, be found in \cite[Proposition~4.6]{Glueck2018}.
\end{proof}

\begin{lemma} 
	\label{lem:closed-ideals}
	Let $E$ be a real or complex Banach lattice and let $f,g,h \in E_+$. 
	The following assertions hold:
	\begin{enumerate}[label=\upshape(\alph*)]
		\item\label{lem:closed-ideals:item:inf} 
		$\overline{E_{f \land g}} = \overline{E_f} \cap \overline{E_g}$.
		
		\item\label{lem:closed-ideals:item:sup} 
		$\overline{E_{f \lor g}} = \overline{E_{f+g}} = \overline{E_f} + \overline{E_g}$.
		
		\item\label{lem:closed-ideals:item:disjoint} 
		If $f$ is disjoint to $h$ and $\overline{E_f} \subseteq \overline{E_{g+h}}$, 
		then actually $\overline{E_f} \subseteq \overline{E_g}$.
	\end{enumerate}
\end{lemma}

\begin{proof}
	\ref{lem:closed-ideals:item:inf} 
	We first observe that, for two arbitrary ideals $I,J \subseteq E$, 
	we have $\overline{I \cap J} = \overline{I} \cap \overline{J}$. 
	Indeed, the inclusion ``$\subseteq$'' is obvious. 
	If conversely $0 \le x \in \overline{I} \cap \overline{J}$, 
	then there exist sequences $(f_n)_{n \in \bbN} \subseteq I$ and $(g_n)_{n\in\bbN} \subseteq J$ which both converge to $x$. 
	After replacing all vectors $f_n$ and $g_n$ with the positive parts of their real parts, 
	we may assume that $f_n,g_n \in E_+$ for all $n$. 
	Now we have $x = \lim_{n \to \infty} f_n \land g_n \in \overline{I \cap J}$. 
	Since $\overline{I} \cap \overline{J}$ is itself an ideal, 
	every vector in $\overline{I} \cap \overline{J}$ 
	is a linear combination of positive vectors in $\overline{I} \cap \overline{J}$, 
	so we have also proved the inclusion ``$\supseteq$''.
	
	It is easy to see that $E_{f\land g} = E_f \cap E_g$; 
	thus it follows from what was said above that 
	$\overline{E_{f \land g}} = \overline{E_f \cap E_g} = \overline{E_f} \cap \overline{E_g}$.
	
	\ref{lem:closed-ideals:item:sup} We have $E_{f\lor g} = E_{f+g} = E_f + E_g$, 
	where the inclusion $E_{f+g} \subseteq E_f + E_g$ follows from the Riesz decomposition property of vector lattices 
	\cite[Propositions~II.1.6 and~II.11.2]{Schaefer1974}. 
	Hence, we conclude that
	\begin{align*}
		\overline{E_{f\lor g}} = \overline{E_{f+g}} = \overline{E_f + E_g} = \overline{E_f} + \overline{E_g},
	\end{align*}
	where the last equality follows from the fact 
	that the sum of two closed ideals in a Banach lattice is again a closed ideal 
	(this result can be found in \cite[Proposition~1.2.2]{Meyer-Nieberg1991} for real Banach lattices; 
	for complex Banach lattices it is a simple consequence of the real case).
	
	\ref{lem:closed-ideals:item:disjoint} 
	For $t \ge 0$ we have
	\begin{align*}
		0 
		\le 
		f \land \big(t(g + h )\big) 
		\le 
		f \land (tg) + f \land (th)
		= 
		f \land (tg)
		\le 
		f;
	\end{align*}
	for the second inequality we used that the infimum operation $\land$ is, 
	when restricted to the positive cone, subadditive in each component 
	\cite[Corollary to Proposition~II.1.6 on p.\,53]{Schaefer1974}.
	As $\overline{E_f} \subseteq \overline{E_{g+h}}$, 
	Lemma~\ref{lem:ideal-inclusion} gives us that $f \land \big(t(g + h )\big) \to f$ as $t \to \infty$.
	Thus, also $f \land (tg) \to f$ as $t \to \infty$.
	By applying Lemma~\ref{lem:ideal-inclusion} once again, 
	we see that $\overline{E_f} \subseteq \overline{E_g}$, as claimed.
\end{proof}

We also need the following simple observation about 
how positive operators preserve inclusion between closed ideals generated by single elements.

\begin{lemma}
	\label{lem:ideal-unter-op}
	Let $E,F$ Banach lattices over the same field.
	\begin{enumerate}[label=\upshape(\alph*)]
		\item\label{lem:ideal-unter-op:item:general} 
		Let $T: E \to F$ be a positive linear operator and let $f,g \in E_+$. 
		If $\overline{E_f} \subseteq \overline{E_g}$, then $\overline{F_{Tf}} \subseteq \overline{F_{Tg}}$.
		
		\item\label{lem:ideal-unter-op:item:higher-powers} 
		Let $T: E \to E$ be a positive lineaer operator, let $f \in E_+$, 
		and assume that $\overline{E_{T^nf}} \supseteq \overline{E_{T^{n-1}f}}$ for some integer $n \ge 1$. 
		Then even $\overline{E_{T^mf}} \supseteq \overline{E_{T^{m-1}f}}$ for all integers $m \ge n$.
	\end{enumerate}
\end{lemma}

\begin{proof}
	\ref{lem:ideal-unter-op:item:general} 
	One has $Tf \in T \overline{E_g} \subseteq \overline{TE_g} \subseteq \overline{F_{Tg}}$. 
	This implies the claim due to Lemma~\ref{lem:ideal-inclusion}.
	
	\ref{lem:ideal-unter-op:item:higher-powers}
	For $m = n+1$ the claim follows from~\ref{lem:ideal-unter-op:item:general}, 
	and for general $m \ge n$ the claim thus follows by induction.
\end{proof}

Now we first prove a version of Theorem~\ref{thm:main-everywhere-bl} for lattice homomorphisms 
before we actually show the theorem itself.

\begin{theorem} 
	\label{thm:lattice-homomorphism}
	Let $T: E \to E$ be a linear lattice homomorphism on a complex Banach lattice $E$.
	Assume that for each $f \in E_+$ there exists an integer $n \ge 1$ 
	such that $\overline{E_{T^nf}} \supseteq \overline{E_{T^{n-1}f}}$ .
	Then $\specPnt(T) \subseteq [0,\infty)$.
\end{theorem}

\begin{proof}
	Suppose for a contradiction that $T$ has an eigenvalue which is not contained in $[0,\infty)$. 
	The point spectrum of every lattice homomorphism is \emph{cyclic},
	i.e., whenever $re^{i\theta}$ is an eigenvalue (for some $r \ge 0$ and $\theta \in [-\pi,\pi)$),
	then so is $re^{in\theta}$ for each $n \in \bbZ$ \cite[Corollary~2 to Proposition~V.4.2 on p.\,324]{Schaefer1974}.
	Thus, we can find a non-zero eigenvalue $\lambda$ of $T$ whose real part satisfies $\re \lambda \le 0$; 
	and since the point spectrum of $T$ is invariant under complex conjugation, we may moreover assume that $\im \lambda \ge 0$. 
	Furthermore, there is no loss of generality in assuming that $\modulus{\lambda} = 1$ 
	(otherwise, we multiply $T$ with an appropriate strictly positive scalar).
	So let $\lambda = e^{i\theta} = \cos \theta + i \sin \theta$, where $\cos \theta \le 0$ and $\sin \theta \ge 0$.
	
	Now, let $0 \not= z = x+iy$ be an eigenvector of $T$ for the eigenvalue $e^{i\theta}$, 
	where $x$ and $y$ are contained in the real part of $E$. 
	It follows from the assumptions of the theorem 
	and from Lemma~\ref{lem:ideal-unter-op}\ref{lem:ideal-unter-op:item:higher-powers} that there exists an integer $n \ge 1$ 
	such that $\overline{E_{T^nf}} \supseteq \overline{E_{T^{n-1}f}}$ for each $f \in \{x^+,x^-,y^+,y^-\}$. 
	Let us define $v := T^{n-1}x$ and $w := T^{n-1}y$. 
	From $T(v + i w) = T T^{n-1}z = e^{i\theta}T^{n-1}z = e^{i\theta}(v + iw)$ we conclude that
	\begin{align*}
		Tv = \cos \theta \; v - \sin \theta \; w 
		\qquad \text{and} \qquad 
		Tw = \cos \theta \; w + \sin \theta \; v.
	\end{align*}
	Those two equalities, along with $\cos \theta \le 0$ and $\sin \theta \ge 0$ 
	and the fact that $T$ is a lattice homomorphism, 
	yield the estimates%
	\footnote{We use the fact that $(a+b)^+ \le a^+ + b^+$ and $(a+b)^- \le a^- + b^-$ 
	for all elements $a,b,c$ of a vector lattice.}
	\begin{align}
		\label{eq:lattice-homomorphism:support-estimate}
		\begin{split}
			& 
			T(v^+) 
			= 
			(Tv)^+ \le -\cos \theta \; v^- + \sin \theta \; w^- \le v^- + w^-, \\
			& 
			T(v^-) 
			= 
			(Tv)^- \le -\cos \theta \; v^+ + \sin \theta \; w^+ \le v^+ + w^+, \\
			& 
			T(w^+) 
			= 
			(Tw)^+ \le - \cos \theta \; w^- + \sin \theta \; v^+ \le w^- + v^+, \\
			& 
			T(w^-) 
			= 
			(Tw)^- \le - \cos \theta \; w^+ + \sin \theta \; v^- \le w^+ + v^-.
		\end{split}
	\end{align}
	By using the first of these four estimates, and again that $T$ is a lattice homomorphism, one concludes
	\begin{align*}
		\overline{E_{v^+}} 
		= 
		\overline{E_{T^{n-1}(x^+)}} 
		\subseteq 
		\overline{E_{T^n(x^+)}} 
		= 
		\overline{E_{T(v^+)}} 
		\subseteq 
		\overline{E_{v^- + w^-}},
	\end{align*}
	As $v^+$ is disjoint to $v^-$ this implies 
	\begin{align*}
		\overline{E_{v^+}} \subseteq \overline{E_{w^-}},
	\end{align*}
	see Lemma~\ref{lem:closed-ideals}\ref{lem:closed-ideals:item:disjoint}. 
	By the same reasoning, the second, third, and fourth estimate in~\eqref{eq:lattice-homomorphism:support-estimate}
	give us
	\begin{align*}
		& 
		\overline{E_{v^-}} 
		\subseteq 
		\overline{E_{T(v^-)}} 
		\subseteq 
		\overline{E_{v^+ + w^+}}, 
		\qquad \text{hence} \qquad 
		\overline{E_{v^-}} 
		\subseteq 
		\overline{E_{w^+}}, \\
		& 
		\overline{E_{w^+}} 
		\subseteq 
		\overline{E_{T(w^+)}} 
		\subseteq 
		\overline{E_{w^- + v^+}}, 
		\qquad \text{hence} \qquad 
		\overline{E_{w^+}} 
		\subseteq 
		\overline{E_{v^+}}, \\
		\text{and} \qquad
		& 
		\overline{E_{w^-}} 
		\subseteq 
		\overline{E_{T(w^-)}} 
		\subseteq 
		\overline{E_{w^+ + v^-}}, 
		\qquad \text{hence} \qquad 
		\overline{E_{w^-}} 
		\subseteq 
		\overline{E_{v^-}}.
	\end{align*}
	Now we can see that $\overline{E_{v^+}} \subseteq \overline{E_{w^-}} \subseteq \overline{E_{v^-}}$, 
	so it follows from Lemma~\ref{lem:closed-ideals}\ref{lem:closed-ideals:item:disjoint} 
	that $\overline{E_{v^+}} \subseteq \overline{E_0} = \{0\}$, so $v^+ = 0$. 
	Similarly, we obtain $v^- = 0$, $w^+ = 0$ and $w^- = 0$. 
	Thus, $0 = T^{n-1}z = e^{i(n-1)\theta}z$, which is a contradiction to $z \not= 0$.
\end{proof}

By employing the celebrated Jacobs--de Leeuw--Glicksberg decomposition theory, 
it is now easy to derive Theorem~\ref{thm:main-everywhere-bl} from Theorem~\ref{thm:lattice-homomorphism}.

\begin{proof}[Proof of Theorem~\ref{thm:main-everywhere-bl}]
	Since $T$ is weakly almost periodic we can employ the Jacobs--de Leeuw--Glicksberg decomposition 
	presented, for instance, in \cite[Section~2.4]{Krengel1985}.
	It gives us a bounded linear projection $P$ on $E$ with the following properties:
	\begin{enumerate}[label=(\alph*)]
		\item 
		The projection $P$ is positive. 
		Hence, as follows from \cite[Proposition~III.11.5]{Schaefer1974}, 
		the range $PE$ is a Banach lattice in its own right 
		with respect to an equivalent norm with positive cone $(PE)_+ := E_+ \cap PE$. 
		
		\item 
		The projection $P$ commutes with $T$ and thus, $T$ leaves both the kernel $\ker P$ and the range $PE$ invariant.
		The restriction $T\restricted{PE}$ is a bijection from $PE$ to $PE$.
		
		Moreover, both $T\restricted{PE}$ and its inverse $(T\restricted{PE})^{-1}$ are positive. 
		Thus, $T\restricted{PE}$ is a lattice isomorphism on $PE$. 
		
		\item 
		The space $PE$ is the closed span of all eigenvectors of $T$ that belong to unimodular eigenvalues. 
	\end{enumerate}
	Due to property~(c), it suffices to prove 
	that $T\restricted{PE}$ does not have any eigenvalues in $\bbT \setminus \{1\}$. 
	To see this, we only have to show that $T\restricted{PE}$ 
	satisfies the assumptions of Theorem~\ref{thm:lattice-homomorphism}.
	According to~(a), $PE$ is a Banach lattice, and according to~(b), $T\restricted{PE}$ is a lattice homomorphism.
	
	So finally let $f \in (PE)_+$. 
	By assumption there is an integer $n \ge 1$ 
	such that $\overline{E_{T^nf}} \supseteq \overline{E_{T^{n-1}f}} \ni T^{n-1}f$.
	By applying Lemma~\ref{lem:ideal-unter-op}\ref{lem:ideal-unter-op:item:general} 
	to the positive operator $P: E \to PE$ we obtain $\overline{(PE)_{T^nf}} \supseteq \overline{(PE)_{T^{n-1}f}}$;
	here we used that $P$ is a projection, and thus $PT^m f = T^m f$ for all $m \ge 0$.
	
	So Theorem~\ref{thm:lattice-homomorphism} is indeed applicable to the operator $T\restricted{PE}$, 
	and we conclude that $\specPnt(T\restricted{PE}) \cap \bbT \subseteq \{1\}$.
\end{proof}

In order to derive Theorem~\ref{thm:main-everywhere} from Theorem~\ref{thm:main-everywhere-bl} 
was used that the inlcusion $\overline{E_	g} \supseteq \overline{E_f}$ is equivalent 
to the almost everywhere inclusion $\supp(g) \supseteq \supp(f)$
if $E$ is an $L^p$-space over a $\sigma$-finite measure space and $p \in [1,\infty)$.
It is important to note that the situation is very different on spaces of continuous functions.
Recall that, for compact Hausdorff space $K$ and a continuous mapping $f: K \to \bbC$ 
the \emph{support} of $f$ is defined as
\begin{align*}
	\supp(f) := \overline{\{x \in K: \; f(x) \not= 0\}}.
\end{align*}
Now consider for instance the space $K = [-1,1]$ and the functions $u,\one \in E := \Cont([-1,1])$, 
where $\one$ denotes the constant function with value $1$ and $u(x) = 1 - \modulus{x}$ for all $x \in [-1,1]$. 
Then both functions $u$ and $\one$ have support $[-1,1]$. 
However, we have $E_{\one} = E$ and thus $\overline{E_{\one}} = E$, while
\begin{align*}
	\overline{E_u} 
	= 
	C_0((-1,1)) := \{f \in E: \; f(-1) = f(1) = 0\}
	\subsetneq 
	E.
\end{align*}
For this reason, we cannot expect Theorem~\ref{thm:main-everywhere} to hold on spaces of continuous functions%
\footnote{We note that, on spaces of continuous functions, inclusions between supports of functions 
need to be treated as set inclusions in the usual sense;
it is not possible to neglect any sets of measure $0$, since there is no canonical choice of a measure.}
(but see, however, Theorem~\ref{thm:everywhere-c-k} and Example~\ref{exa:nagler} below). 
Here is a concrete counterexample:

\begin{example}	
	\label{ex:weakly-expanding}
	Let $E = \Cont([-1,1])$. 
	There exists a positive finite rank operator $T \in \calL(E)$ which satisfies $\spr(T) = \norm{T} = 1$ 
	and which has the following properties:
	\begin{enumerate}[label=(\alph*)]
		\item 
		One has $\supp(Tf) = [-1,1]$ for every non-zero $f \in E_+$; 
		in particular, $\supp(Tf) \supseteq \supp(f)$ for all $f \in E_+$.
		
		\item 
		The point spectrum of $T$ contains the number $-1$.
	\end{enumerate}
	Indeed, let $u,v,w \in E_+$ be given by
	\begin{align*}
		u(x) = 1 - \modulus{x}, 
		\qquad 
		v(x) = |x| \one_{[-1,0]}(x), 
		\qquad 
		w(x) = |x| \one_{[0,1]}(x)
	\end{align*}
	for all $x \in [-1,1]$. 
	We define $T \in \calL(E)$ by
	\begin{align*}
		Tf 
		= 
		\frac{1}{2}\int_{[-1,1]} f \dx x \cdot u + f(1) \, v + f(-1) \, w
	\end{align*}
	for all $f \in E$. 
	Clearly, $T$ has finite rank, is positive, and satisfies $T \one = \one$; 
	the latter two properties in turn imply $\spr(T) = \norm{T} = 1$.
	
	Since $\supp(u) = [-1,1]$, it follows that $\supp(Tf) = [-1,1]$ for every non-zero $f \ge 0$; this proves~(a). 
	On the other hand, $T(v-w) = -v+w$, so $-1$ is an eigenvalue of $T$; this proves~(b).
\end{example}

\begin{remark}
	In \cite[Exercise~8(c) on p.\,353]{Schaefer1974} the following assertion is claimed: 
	
	Let $E$ be a Banach lattice, let $T \in \calL(E)$ be positive, 
	let us say for the sake of convenience with spectral radius $1$, 
	and assume that 
	$$\sup_{\lambda \in (1,\infty)} (\lambda - 1) \norm{(\lambda-T)^{-1}} < \infty.$$ 
	If for every non-zero $f \in E_+$ there exists an integer $n \ge 1$ such that 
	$T^nf$ is a \emph{weak order unit} of $E$ 
	(see \cite[p.\,55]{Schaefer1974} for definition of this notion),
	then $\specPnt(T) \cap \bbT \subseteq \{1\}$.
	
	Examples~\ref{ex:weakly-expanding} shows that this claim is not correct.
\end{remark}

The problem in Example~\ref{ex:weakly-expanding} is that the support of a continuous function $f$, 
defined as the closure of the set where $f$ does not vanish, 
is not well-suited to the closed ideals in $\Cont(K)$ generated by $f$. 
However, the situation becomes much better if we consider the \emph{open support} 
of functions. 
For a topological space $X$ and a continuous scalar-valued function $f$ on $X$ the open support 
is the (open) set
\begin{align*}
	\opensupp(f) := \{x \in X: \; f(x) \not= 0\}.
\end{align*}
On compact spaces, inclusions of the type $\overline{E_f} \supseteq \overline{E_g}$ can be characterized 
by means of the open supports of $f$ and $g$ (see Lemma~\ref{lem:ideal-inclusion-c-k} below);
thus we get the following version of Theorem~\ref{thm:main-everywhere-bl} on $\Cont(K)$.

\begin{theorem}
	\label{thm:everywhere-c-k}
	Let $K$ be a compact Hausdorrf space and let 
	$T: \Cont(K) \to \Cont(K)$ be a positive linear operator which is weakly almost periodic. 
	Assume that, for each $0 \le f \in \Cont(K)$, 
	there exists an integer $n \ge 1$ such that $\opensupp(T^nf) \supseteq \opensupp(T^{n-1}f)$. 
	Then $\specPnt(T) \cap \bbT \subseteq \{1\}$.
\end{theorem}

The theorem is an immediate consequence of Theorem~\ref{thm:main-everywhere-bl} and the following lemma.

\begin{lemma}
	\label{lem:ideal-inclusion-c-k}
	Let $K$ be a compact Hausdorff space and let $0 \le f,g \in E := \Cont(K)$.
	Then $\overline{E_f} \subseteq \overline{E_g}$ if and only if $\opensupp(f) \subseteq \opensupp(g)$.
\end{lemma}

\begin{proof}
	\rightProof
	According to Lemma~\ref{lem:ideal-inclusion} we have $f \land (tg) \to f$ in norm as $t \to \infty$.
	Let $x \in \opensupp(f)$. 
	Then $f(x) > 0$, so we conclude from $f(x) \land (tg(x)) \to f(x)$ as $t \to \infty$ that $g(x) > 0$,
	i.e., $x \in \opensupp(g)$.
	
	\leftProof
	Consider the increasing net $\big(f \land (tg)\big)_{t \in (0,\infty)}$ in $\Cont(K)$. 
	It follows from the inclusion $\opensupp(f) \subseteq \opensupp(g)$ 
	that this net converges pointwise to the continuous function $f$.
	Due to Dini's theorem the convergence is automatically uniform, 
	so Lemma~\ref{lem:ideal-inclusion} implies that $\overline{E_f} \subseteq \overline{E_g}$.
\end{proof}

Very easy examples (for instance on the Euclidean space $\bbR$) 
show that Lemma~\ref{lem:ideal-inclusion-c-k} does not remain true 
if we replace the compact space $K$ with, for instance, a locally compact space.

We close this subsection by showing that Theorem~\ref{thm:main-everywhere-bl} 
can be interpreted as a strong generalization of the following result 
of Nagler \cite[Theorem~1]{Nagler2015} for finite rank operators.
For a vector $x$ in a Banach space $E$ and a functional $\alpha \in E'$
we use the notation $x \otimes \varphi$ for the linear operator $E \to E$ given by
\begin{align*}
	(x \otimes \alpha) f := \langle \alpha, f \rangle x
\end{align*}
for all $f \in E$.

\begin{example}[Nagler]
	\label{exa:nagler}
	Let $E$ be a complex Banach lattice, 
	let $e_1, \dots, e_n \in E_+$ and $\alpha_1, \dots \alpha_n \in E'_+$, 
	and set $e := e_1 + \dots + e_n$.
	Assume that $\langle \alpha_j, e \rangle = 1$ and $\langle \alpha_j, e_j \rangle > 0$ 
	for each $j \in \{1, \dots, n\}$.
	Then the finite rank operator 
	\begin{align*}
		T 
		:= 
		e_1 \otimes \alpha_1 + \dots e_n \otimes \alpha_n
		: 
		E \to E
	\end{align*}
	is power bounded, has $e$ as a fixed point, 
	and satisfies $\specPnt(T) \cap \bbT \subseteq \{1\}$.
\end{example}

\begin{proof}
	One readily checks that $Te = e$. 
	Thus, the orbits of $e_1, \dots, e_n$ under the powers of $T$ are contained in $[0,e]$;
	and since $e_1, \dots, e_n$ span the range of $T$ it follows that $T$ is power bounded. 
	
	As $T$ has finite-rank, we conclude in particular that $T$ is weakly almost periodic. 
	In order to conclude $\specPnt(T) \cap \bbT \subseteq \{1\}$ 
	we show that the assumption $\overline{E_{T^nf}} \supseteq \overline{E_{T^{n-1}f}}$ 
	in Theorem~\ref{thm:main-everywhere-bl}
	is satisfied for every $f \in E_+$ for $n=2$.
	Fix $f \in E_+$ and let
	\begin{align*}
		A := \big\{j \in \{1, \dots, n\} : \; \langle \alpha_j, f \rangle > 0 \big\}.
	\end{align*}
	With the notation $e_A := \sum_{j \in A} e_j$
	it follows that $c e_A \le Tf \le d e_A$ for some numbers $c,d > 0$, 
	so $E_{Tf} = E_{e_A}$.
	By letting $a > 0$ be the smallest of the numbers 
	$\langle \alpha_1, e_1 \rangle, \dots, \langle \alpha_n, e_n \rangle$, we finally get
	\begin{align*}
		T^2 f 
		\ge 
		c T e_A 
		\ge 
		c \sum_{j \in A} Te_j 
		\ge 
		ca \sum_{j \in A} e_j 
		= 
		ca e_A,
	\end{align*}
	so $E_{T^2f} \supseteq E_{e_A} = E_{Tf}$, 
	and hence $\overline{E_{T^2f}} \supseteq \overline{E_{Tf}}$.
\end{proof}

Nagler considered Banach function spaces $E$ in \cite[Theorem~1]{Nagler2015} 
and assumed $e$ to be the constant function with value $1$, 
but this assumption is not essential for the result.
(In \cite[Section~1.2]{Nagler2015} 
the assumption $\norm{\alpha_j} = 1$ for each $j \in \{1, \dots, n\}$ is also listed; 
but this assumption is, in conjunction with the assumption that $\langle \alpha_j, e \rangle = 1$
for the constant $1$-function $e$, 
rather restrictive, and is not needed for the result to hold.)

\begin{remark}
	It is illuminating to compare Example~\ref{exa:nagler} 
	to the counterexample in~\ref{ex:weakly-expanding}. 
	The operator $T$ in Example~\ref{ex:weakly-expanding} 
	can also be written in the form
	\begin{align*}
		T = e_1 \otimes \alpha_1 + e_2 \alpha_2 + e_3 \otimes \alpha_3,
	\end{align*}
	where $e_1 = u$, $e_2 = v$, $e_3 = w$ (and thus $e = e_1+e_2+e_3 = \one$), 
	and where $\alpha_1$ is the integration against $1/2$ times the Lebesgue measure on $[-1,1]$, 
	$\alpha_2$ is the point evaluation at $1$, and $\alpha_3$ is the point evaluation at $-1$.
	
	One also has $\langle \alpha_j, e \rangle = 1$ in this example; 
	the reason why one is not in the situation of Example~\ref{exa:nagler} 
	is that $\langle \alpha_2, e_2 \rangle = v(1) = 0$ 
	and $\langle \alpha_3, e_3 \rangle = w(-1) = 0$.
\end{remark}

\subsection{Long-term behaviour}
\label{subsection:primitivity-1:asymptotics}

We now use the results of the previous subsection 
to analyze the behaviour of the powers $T^n$ as $n \to \infty$.

A bounded linear operator $T$ on a Banach lattice $E$ is called \emph{AM-compact} 
if it maps every order interval to a relatively compact set.
Obviously, every compact operator is AM-compact. 
Other examples of AM-compact operators are discussed after the following theorem.

\begin{theorem} 
	\label{thm:everywhere-convergence}
	Let $T: E \to E$ be a power bounded linear operator on a complex Banach lattice $E$,
	and assume that the fixed space $\ker(1-T)$ contains a quasi-interior point $h$ of $E_+$.
	Assume moreover that the operator $T^{n_0}$ is AM-compact for some integer $n_0 \ge 1$ 
	and that, for every $f \in E_+$, there exists an integer $n \ge 1$ 
	such that $\overline{E_{T^nf}} \supseteq \overline{E_{T^{n-1}f}}$. 
	Then $T^k$ converges strongly as $k \to \infty$.
\end{theorem}

\begin{proof}
	Since the order interval $[0,h]$ is invariant under $T$, 
	it follows from the AM-compactness of $T^{n_0}$ that the orbit of each point $f \in [0,h]$
	under $\{T^n: \; n \ge 0\}$ is relatively compact in $E$;
	thus, the same is true for every $f$ in the span $E_h$ of $[0,h]$.
	
	Since $E_h$ is dense in $E$ and $T$ is power bounded, 
	it even follows that the orbit of every $f \in E$ is relatively compact 
	\cite[Corolary~A.5(a)]{Engel2000}.
	In particular, each orbit is relatively weakly compact, 
	i.e., the operator $T$ is weakly almost periodic.
	So Theorem~\ref{thm:main-everywhere-bl} implies that $\specPnt(T) \cap \bbT \subseteq \{1\}$.
	
	Also due to the AM-compactness of $T^{n_0}$ and the existence of the quasi-interior fixed point, 
	\cite[Theorem~6.1]{GlueckHaase2019} is applicable.
	Part~(c) of this result tells us that $T^n$ converges strongly as $n \to \infty$ 
	since $T$ does not have eigenvalues in $\bbT \setminus \{1\}$.
\end{proof}

Next we state a more explicit version of the previous theorem on $L^p$-spaces. 
Let $(\Omega,\mu)$ be a $\sigma$-finite measure space and $p \in [1,\infty)$; 
we consider the (complex-valued) space $L^p := L^p(\Omega,\mu)$. 
A positive linear operator $T: L^p \to L^p$ is called an \emph{integral operator} 
or \emph{kernel operator} if there exists a measurable function $k: \Omega \times \Omega \to [0,\infty)$ 
-- the so-called \emph{integral kernel} of $T$ -- 
such that the following property holds for every $f \in L^p$:
for almost every $\omega \in \Omega$, the function $k(\omega,\argument)f$ is an $L^1(\Omega,\mu)$ and one has
\begin{align*}
	(Tf)(\omega) = \int_\Omega k(\omega, \tilde \omega) f(\tilde \omega) \dx \tilde \omega.
\end{align*}
Every positive integral operator is AM-compact; 
see for instance \cite[Appendix~A]{GerlachGlueck2019} or \cite[Appendix~A]{GlueckHaase2019} for more details.

\begin{corollary} 
	\label{cor:everywhere-convergence-l-p}
	Let $p \in [1,\infty)$, let $(\Omega,\mu)$ be a $\sigma$-finite measure space 
	and set $L^p := L^p(\Omega,\mu)$. 
	Let the positive linear operator $T: L^p \to L^p$ be an integral operator 
	whose integral kernel $k: \Omega \times \Omega \to [0,\infty)$ has the following property: 
	for every measurable set $M \subseteq \Omega$ with non-zero measure we have 
	$\int_{M \times M} k \dx \mu \otimes \mu > 0$.
	
	If $T$ is power bounded and has a fixed point which is strictly positive almost everywhere on $\Omega$, 
	then $T^k$ converges strongly as $k \to \infty$.
\end{corollary}

\begin{proof}
	As pointed out before the corollary, $T$ is AM-compact.
	So, in order to apply Theorem~\ref{thm:everywhere-convergence}, 
	it suffices to check that $\supp Tf \supseteq \supp f$ 
	(in the almost everywhere sense) for each $0 \le f \in L^p$.
	
	Fix $0 \le f \in L^p$
	(more precisely, let $f$ be a representative of a positive vector in $L^p$; 
	we may choose $f$ such that $f(\omega) \ge 0$ for all $\omega \in \Omega$) 
	and let $g = Tf$. 
	We have
	\begin{align*}
		g(\omega) = \int_\Omega k(\omega, \tilde \omega) f(\tilde \omega) \dx \mu(\tilde \omega)
	\end{align*}
	for almost all $\omega \in \Omega$;
	we may even modify $k$ and $g$ on a null set of $\Omega \times \Omega$ and $\Omega$, respectively, 
	such that this formula holds for all $\omega \in \Omega$ (where the modification of $k$ might depend on $f$). 	
	Set 
	\begin{align*}
		S_f := \{\omega \in \Omega: \; f(\omega) > 0\}
		\qquad \text{and} \qquad 
		S_g := \{\omega \in \Omega: \; g(\omega) > 0\}.
	\end{align*}
	Since $\int_{\Omega \setminus S_g} g(\omega)\dx \mu(\omega) = 0$, 
	we conclude that $k(\omega, \tilde \omega) f(\tilde \omega) = 0$ 
	for almost all $(\omega, \tilde \omega) \in (\Omega \setminus S_g) \times \Omega $. 
	Hence, $k(\omega, \tilde \omega) = 0$ for almost all $(\omega,\tilde \omega) \in (\Omega \setminus S_g) \times S_f$. 
	In particular, we have $k(\omega,\tilde \omega) = 0$ 
	for almost all $(\omega, \tilde \omega) \in (S_f \setminus S_g) \times (S_f \setminus S_g)$.
	
	By our assumption on $k$ this implies that $S_f \setminus S_g$ has measure $0$, 
	so we indeed have $\supp Tf = \supp g \supseteq \supp f$.
\end{proof}

The condition on the integral kernel $k$ in Corollary~\ref{cor:everywhere-convergence-l-p} means, 
in a sense, that $k$ is non-zero close to the diagonal of $\Omega \times \Omega$. 
We make this more explicit in the following version of the corollary which is stated in terms of topological spaces. 
Recall that a topological space $\Omega$ is called \emph{Lindelöf} 
if every open cover of $\Omega$ admits an at most countable subcover. 
Of course, every compact topological space is Lindelöf; 
moreover, a metric space is Lindelöf if and only if it is separable. 
Hence, the following corollary is applicable to a large variety of spaces.

\begin{corollary} \label{cor:lindeloef}
	Let $p \in [1,\infty)$ and let $\Omega$ be a topological space which is Lindelöf. 
	Let $\mu $ be a $\sigma$-finite measure (with values in $[0,\infty)$) 
	defined on the Borel $\sigma$-algebra on $\Omega$ and set $L^p := L^p(\Omega,\mu)$.
	
	Let $T: L^p \to L^p$ be a positive linear operator which is an integral operator 
	and whose integral kernel $k: \Omega \times \Omega \to [0,\infty)$ has the following property: 
	there exists an open set $U \subseteq \Omega \times \Omega$ which contains the diagonal 
	$\{(\omega,\omega) \in \Omega \times \Omega: \; \omega \in \Omega\}$ 
	such that $k(\omega, \tilde \omega) > 0$ for almost all $(\omega, \tilde \omega) \in U$.
	
	If $T$ is power bounded and has a fixed point which is strictly positive almost everywhere on $\Omega$, then $T^k$ converges strongly as $k \to \infty$.
\end{corollary}

Note that we do not assume that measure $\mu$ in the corollary to be strictly positive 
in the sense that $\mu(V) > 0$ for every non-empty open set $V$.
For the proof we need a little lemma about Lindelöf spaces:

\begin{lemma} 
	\label{lem:lindeloef-diagonal}
	Let $\Omega$ be a topological space which is Lindelöf 
	and let $U \subseteq \Omega \times \Omega$ be an open set 
	which contains the diagonal $\Delta := \{(\omega,\omega) \in \Omega \times \Omega: \; \omega \in \Omega\}$.
	
	Then there exists an at most countable system $\calB$ of open sets $B \subseteq \Omega$ 
	such that $\calB$ covers $\Omega$ and such that $B \times B \subseteq U$ for all $B \in \calB$.
\end{lemma}

\begin{proof}
	Let $\omega \in \Omega$. 
	Since $U$ is an open neighbourhood of $(\omega,\omega)$ in $\Omega \times \Omega$, 
	we can find an open neighbourhood $B_\omega$ of $\omega$ in $\Omega$ 
	such that $B_\omega \times B_\omega \subseteq U$.
	
	We have $\bigcup_{\omega \in \Omega} B_\omega = \Omega$; 
	using that $\Omega$ is Lindelöf, 
	we can thus find an at most countable subsystem $\calB$ of $\{B_\omega: \; \omega \in \Omega\}$ 
	which also covers $\Omega$. 
	This proves the assertion.
\end{proof}

\begin{proof}[Proof of Corollary~\ref{cor:lindeloef}]
	We only have to check that $k$ satisfies 
	the assumption stated in Corollary~\ref{cor:everywhere-convergence-l-p}, 
	so let $M \subseteq \Omega$ be a measurable set of non-zero measure. 
	
	Since $\Omega$ is Lindelöf, 
	Lemma~\ref{lem:lindeloef-diagonal} shows that we can find an at most countable open cover $\calB$ of $\Omega$ 
	such that $B \times B \subseteq U$ for each $B \in \calB$.
	There exists a set $B \in \calB$ such that $B \cap M$ has non-zero measure. 
	For this set $B$, the square $S := (B \cap M) \times (B \cap M)$ has non-zero measure in $\Omega \times \Omega$ 
	and it is contained in $U$, so $k$ is strictly positive almost everywhere on $S$. 
	Thus,
	\begin{align*}
		\int_{M \times M} k \dx \mu \otimes \mu \ge \int_S k \dx \mu \otimes \mu > 0,
	\end{align*}
	so the assumptions of Corollary~\ref{cor:everywhere-convergence-l-p} are indeed satisfied.
\end{proof}

Let us illustrate Corollary~\ref{cor:lindeloef} with the following simple example:

\begin{example} 
	\label{ex:integral-operator-on-unit-interval}
	Endow the unit interval $[0,1]$ with the Borel-$\sigma$-algebra and the Lebesgue measure 
	and let $k: [0,1] \times [0,1] \to [0,\infty)$ be a measurable function such that 
	\begin{align*}
		& 
		\int_{[0,1]} k(x,y) \dx x = 1 
		\quad \text{for almost all } y \in [0,1] 
		\\
		\text{and} \qquad 
		& 
		\int_{[0,1]} k(x,y) \dx y = 1 
		\quad \text{for almost all } x \in [0,1].
	\end{align*}
	Moreover, let $\delta > 0$ and assume that $k$ is strictly positive almost everywhere 
	on the diagonal strip $\{(x,y) \in [0,1]^2: \; \modulus{x-y} < \delta\}$. 
	For each $0 \le f \in L^1 := L^1([0,1])$ we have
	\begin{align*}
		\int_{[0,1]} \int_{[0,1]} k(x,y) f(x) \dx x \dx y = \norm{f}_{L^1}.
	\end{align*}
	Thus, $k(x,\argument)f(\argument)$ is in $L^1$ for almost all $x \in [0,1]$, 
	the function $\int_{[0,1]} k(\argument,y)f(y)\dx y$ is again in $L^1$, 
	and its norm equals the norm of $f$. 
	Hence,
	\begin{align*}
		L^1 \ni f \mapsto Tf := \int_{[0,1]} k(\argument,y) f(y) \dx y \in L^1
	\end{align*}
	defines a positive linear operator $T: L^1 \to L^1$ of norm $1$. 
	Moreover, one readily checks that $T\one_{[0,1]} = \one_{[0,1]}$. 
	Corollary~\ref{cor:lindeloef} thus implies that $T^k$ converges strongly as $k \to \infty$.
\end{example}

We also have the following consequence of 
Theorem~\ref{thm:everywhere-convergence} on sequence spaces:

\begin{corollary} 
	\label{cor:convergence-small-l-p}
	Let $p \in [1,\infty)$ and set $\ell^p := \ell^p(\bbN)$. 
	Let $T: \ell^p \to \ell^p$ be a positive and power bounded operator 
	and assume that $T$ has a fixed point $f_0$ such that $f_0(\omega) > 0$ for all $\omega \in \bbN$. 
	Assume moreover that $\langle e_j, T e_j \rangle > 0$ for each canonical unit vector $e_j$.
	
	Then $T^k$ converges strongly as $k \to \infty$.
\end{corollary}

\begin{proof}
	We clearly have $\supp (Tf) \supseteq \supp f$ for each $0 \le f \in \ell^p$. 
	Moreover, it is not difficult to check that every order interval in $\ell^p$ is compact 
	and thus, every bounded linear operator on $\ell^p$ is AM-compact. 
	Hence, the assertion follows from Theorem~\ref{thm:everywhere-convergence}.
\end{proof}

Note that we can also replace $\ell^p$ with the space $c_0$ in the above corollary 
and the same conclusion remains true (with the same proof).

By now we have explored several convergence results for the powers of kernel operators 
or, more generally, AM-compact operators.
We conclude this chapter by considering the case where (a power of) $T$ only dominates a non-zero AM-compact operator.
Operators with this property 
-- and especially the special case of so-called \emph{partial integral operators} that dominate a non-zero integral operator -- 
frequently occur in the literature. 
See for instance \cite[Sections~3 and~4]{Gerlach2013} 
and \cite[Theorem~1 on p.\,247]{Rudnicki1995} for an analysis of 
spectral and asymptotic properties of a (sub)class of such operators, 
and \cite[Section~5]{GerlachGlueckKernel} for a characterization of this operator class.
A convergence result for $C_0$-semigroups on $L^1$ which contain a partial integral operator can, 
for instance, be found in \cite[Theorems~1 and~2]{Pichor2000}; 
this result turned out to have numerous applications to evolution equations in mathematical biology.

For a positive operator $T$ on a Banach lattice $E$ we call a vector $x \in E$ 
a \emph{fixed point} of $T$ if $Tx = x$.
We call a vector $x \in E_+$ a \emph{super-fixed point} of $T$ if $Tx \ge x$.

\begin{theorem} 
	\label{thm:convergence-partial-kernel}
	Let $E$ be a Banach lattice with order continuous norm 
	and let $T: E \to E$ be a power bounded positive linear operator on $E$ 
	whose fixed space $\ker(1-T)$ contains a quasi-interior fixed point of $E_+$. 
	Assume that the following assertions are fulfilled:
	\begin{enumerate}[label=\upshape(\arabic*)]
		\item 
		Every super-fixed point in $E_+$ of $T$ is actually a fixed point of $T$.
		
		\item 
		For every non-zero $0 \le g \in \fix T$ there exists an integer $n \ge 0$ 
		and an AM-compact operator $K \in \calL(E)$ such that $0 \le K \le T^n$ and $Kg \not= 0$.
		
		\item 
		For every $0 \le f \in E$ there exists an integer $n \ge 1$ 
		such that $\supp(T^nf) \supseteq \supp(T^{n-1}f)$.
	\end{enumerate}
	Then $T^k$ converges strongly as $k \to \infty$.
\end{theorem}

Sufficient conditions for every super-fixed point in $E_+$ to actually be a fixed point
can, for instance, be found in \cite[Proposition~3.11]{GerlachGlueck2019}.

\begin{proof}[Proof of Theorem~\ref{thm:convergence-partial-kernel}]
	Due to the assumptions~(1) and~(2) and the existence of a quasi-interior fixed point 
	it follows that the operator semigroup $(T^n)_{n \in \bbN_0}$ satisfies the \emph{standard assumptions} 
	in \cite[Section~6]{GlueckHaase2019};
	this is proved in \cite[Lemma~3.10]{GerlachGlueck2019}
	(the divisibility assumption in Theorem~3.9 in this reference 
	is not used in the proof of Lemma~3.10).
	
	Therefore, it follows from \cite[Theorem~3.1(a) and Remark~2.6(2)]{GlueckHaase2019} 
	that every orbit of the powers of $T$ is relatively compact in $E$; 
	in particular, every orbit is relatively weakly compact, so $T$ is weakly almost periodic.
	So we can apply Theorem~\ref{thm:main-everywhere-bl} to derive that $\specPnt(T) \cap \bbT \subseteq \{1\}$, 
	and thus the claimed convergence follows from \cite[Theorem~6.1(c)]{GlueckHaase2019}.
\end{proof}

Theorem~\ref{thm:convergence-partial-kernel} is, in a sense, 
an analogous result to Theorem~\ref{thm:everywhere-convergence}. 
Thus, one can derive similar consequences for Theorem~\ref{thm:convergence-partial-kernel} 
as we did for Theorem~\ref{thm:everywhere-convergence} above. 
However, it would probably not be particularly illuminating to repeat the same arguments again, 
so we omit the details here.

When combining our triviality results on the point spectrum 
with spectral analysis of quasi-compact operators 
(which is, for instance, presented in \cite[Theorem~2.8 on pp.\,91--92]{Krengel1985} or \cite{Nagler2018}) 
one can also obtain sufficient conditions for operator norm convergence 
of powers of positive operators.
As these arguments do not seem to add many insights to the main theme of this article though, 
we do not discuss this in detail.

\section{Irreducibe Operators that partially increase the support of functions} 
\label{section:conditions-for-primitivity-2}

\subsection{Point spectrum}
\label{subsection:conditions-for-primitivity-2:point-spectrum}

In some cases it can happen that the property of an operator to increase the support of functions 
is not satisfied everywhere on the underlying space, but only on a part of the space. 
If the operator is irreducible, we can still derive similar results as before. 
This is the purpose of this section.

A positive linear operator $T$ on a Banach lattice $E$ is called \emph{irreducible} if, 
for each non-zero $f \in E_+$ and each non-zero functional $\varphi \ge 0$ in the norm dual $E'$ of $E$, 
there exists an integer $n \ge 0$ such that $\langle \varphi, T^n f \rangle > 0$.
Equivalently, $\{0\}$ and $E$ are the only closed $T$-invariant ideals in $E$.
A vector subspace $B$ of a Banach lattice $E$ is called a \emph{projection band} 
if there exists a bounded linear projection $P$ on $E$ which has range $B$ 
and satisfies $0 \le Pf \le f$ for all $f \in E_+$.
For instance, on an $L^p$-space the multiplication with the indicator function of a measurable set 
is a band projection.

\begin{theorem}
	\label{thm:irred-bl}
	Let $E$ be a complex Banach lattice and let $T: E \to E$ be a positive linear operator 
	which is weakly almost periodic.
	Assume that $T$ is irreducible and that there exists a projection band $\{0\} \not= B \subseteq E$ 
	such that $\overline{E_{Tf}} \supseteq \overline{E_f}$ for each $0 \le f \in B$. 
	Then $\specPnt(T) \cap \bbT \subseteq \{1\}$.
\end{theorem}

\begin{proof}
	We use similar arguments as in the proofs of 
	Theorems~\ref{thm:main-everywhere-bl} and~\ref{thm:lattice-homomorphism}. 
	Since $T$ is weakly almost periodic, we can employ the Jacobs--de Leeuw--Glicksberg decomposition 
	\cite[Section~2.4]{Krengel1985}
	to obtain a bounded linear projection $P$ on $E$ 
	which has the same properties as listed in the proof of Theorem~\ref{thm:main-everywhere-bl}:
	\begin{enumerate}[label=(\alph*)]
		\item 
		The projection $P$ is positive, 
		and thus range $PE$ is a complex Banach lattice with respect to an equivalent norm,
		and with positive cone $(PE)_+ := E_+ \cap PE$.
		
		\item 
		The projection $P$ commutes with $T$, so $T$ leaves $\ker P$ and $PE$ invariant;
		moreover, the restriction $T\restricted{PE}$ is a lattice isomorphism on $PE$.
		
		\item 
		The range $PE$ is the closed span of the eigenvectors of $T$ belonging to unimodular eigenvalues. 
	\end{enumerate}
	Assume now that the assertion of the theorem is false.
	Then there exists an unimodular eigenvalue $e^{i\theta} \not= 1$ of $T$ (with some $\theta \in [-\pi,\pi)$), 
	and hence, $e^{i\theta} \in \specPnt(T\restricted{PE})$.
	Since the letter point spectrum is cyclic 
	according to \cite[Corollary~2 to Proposition~V.4.2]{Schaefer1974}, 
	we can even assume that $\re e^{i\theta} \le 0$. 
	Moreover, since the point spectrum of $T\restricted{PE}$ is invariant under complex conjugation, 
	we can also assume that $\im e^{i\theta} \ge 0$. 
	
	Thus, $e^{i\theta} = \cos \theta + i\sin \theta$, where $\cos \theta \le 0$ and $\sin \theta \ge 0$. 
	Let $z = x+iy \in PE \setminus \{0\}$ be an eigenvector of $T\restricted{PE}$ for the eigenvalue $e^{i\theta}$ 
	(where $x,y$ are contained in the real parts of $E$ and $PE$).
	
	Consider the vector $\modulus{z}$ (where the modulus is computed in the Banach lattice $E$, not in $PE$)%
	\footnote{
		We will actually see later in the proof that the modulus operation is the same in both lattices 
		$E$ and $PE$, but at this point we do not know this, yet.
	}%
	. 
	We have $T \modulus{z} \ge \modulus{Tz} = \modulus{z}$. 
	Since $T$ is irreducible and power bounded, it follows from the inequality 
	$T \modulus{z} \ge \modulus{Tz} = \modulus{z}$ for the non-zero vector $\modulus{z}$  
	that there exists a strictly positive function $\varphi \in E'_+$ which is a fixed point of $T'$; 
	this is explained in the proof of \cite[Proposition~3.11(c)]{GerlachGlueck2019}.
	Moreover, the same reference shows that $\modulus{z}$ is actually a fixed point of $T$.
	Thus, the principal ideal $E_{\modulus{z}}$ is $T$-invariant and hence, so is its closure. 
	As $T$ is irreducible, it follows that $\overline{E_{\modulus{z}}} = E$, 
	so $\modulus{z}$ is a quasi-interior point of $E_+$.
	
	Since $P$ is, by the Jacobs--de Leeuw--Glicksberg theory, 
	contained in the weak operator closure of $\{T^n: \; n \ge 0\}$, 
	it readily follows that the strictly positive functional $\varphi$ is also a fixed point of $P'$. 
	This implies that $P$ is \emph{strictly positive} in the sense that 
	$Pf$ is non-zero for every non-zero $f \in E_+$.
	Therefore, $PE$ is even a sublattice of $E$ (this follows from \cite[Proposition~III.11.5]{Schaefer1974}) 
	and thus, the lattice operations in the Banach lattice $PE$ coincide with the lattice operations in $E$.
	
	Similarly as in the proof of Theorem~\ref{thm:lattice-homomorphism} we now observe that
	\begin{align*}
		Tx = \cos \theta \; x - \sin \theta \; y \qquad \text{and} \qquad Ty = \cos \theta \; y + \sin \theta \; x.
	\end{align*}
	By using that $T\restricted{PE}$ is a lattice homomorphism, we conclude in the same way as in the proof of Theorem~\ref{thm:lattice-homomorphism} that
	\begin{align*}
		& T(x^+) \le x^- + y^- \qquad \text{and} \qquad T(x^-) \le x^+ + y^+, \\
		\text{as well as} \qquad
		& T(y^+) \le y^- + x^+ \qquad \text{and} \qquad T(y^-) \le y^+ + x^-.
	\end{align*}
	Now, let $Q \in \calL(E)$ denote the band projection onto $B$. Then we have
	\begin{align*}
		\overline{E_{Qx^+}} \subseteq \overline{E_{TQx^+}} \subseteq \overline{E_{Tx^+}} \subseteq \overline{E_{x^- + y^-}},
	\end{align*}
	where the first inclusion follows from our assumption on $T$. 
	Since $Qx^+$ is dominated by $x^+$ and thus disjoint to $x^-$, 
	we conclude from Lemma~\ref{lem:closed-ideals}\ref{lem:closed-ideals:item:disjoint} 
	that $\overline{E_{Qx^+}} \subseteq \overline{E_{y^-}}$. 
	Moreover, as $Qx^+$ is contained in the band $B$, so is $\overline{E_{Qx^+}}$, 
	and hence we obtain
	\begin{align*}
		\overline{E_{Qx^+}} 
		= 
		Q \overline{E_{Qx^+}} 
		\subseteq 
		Q \overline{E_{y^-}} 
		\subseteq 
		\overline{QE_{y^-}} 
		\subseteq 
		\overline{E_{Qy^-}}.
	\end{align*}
	By the same reasoning we can see that
	\begin{align*}
		\overline{E_{Qx^-}} 
		\subseteq 
		\overline{E_{Qy^+}} 
		\qquad \text{as well as} \qquad 
		\overline{E_{Qy^+}} \subseteq \overline{E_{Qx^+}} 
		\qquad \text{and} \qquad 
		\overline{E_{Qy^-}} \subseteq \overline{E_{Qx^-}}.
	\end{align*}
	Hence, we have $\overline{E_{Qx^+}} \subseteq \overline{E_{Qy^-}} \subseteq \overline{E_{Qx^-}}$. 
	However, $Qx^+$ and $Qx^-$ are disjoint, 
	so it follows from Lemma~\ref{lem:closed-ideals}\ref{lem:closed-ideals:item:disjoint} 
	that $\overline{E_{Qx^+}} = \{0\}$ and thus, $Qx^+ = 0$. 
	Using the same argument again, we can also see that $Qx^- = 0$, 
	as well as $Qy^+ = 0$ and $Qy^- = 0$.
	
	Yet, we have $\modulus{z} \le \modulus{x} + \modulus{y} = x^+ + x^- + y^+ + y^-$, 
	so $0 \le Q\modulus{z} \le 0$, i.e., $Q\modulus{z} = 0$. 
	As we have noted above that $\modulus{z}$ is a quasi-interior point of $E_+$, 
	the equality $Q\modulus{z} = 0$ implies that $Q = 0$. 
	This contradicts the assumption $B \not= \{0\}$.
\end{proof}

As in Subsection~\ref{subsection:primitivity-1:point-spectrum}, 
we will now derive a corollary of the previous theorem 
which does only require $T$ to be power bounded rather than almost weakly periodic.
We recall that a Banach lattice $E$ is said to have \emph{order continuous norm} 
if every decreasing net in $E_+$ with infimum $0$ is norm convergent to $0$. 
Every KB-space (and thus, in particular, every $L^p$-space for $1 \le p < \infty$) has order continuous norm;
the space $c_0$ of sequence that converge to $0$ (endow with the $\infty$-norm) 
is an example of a Banach lattice that has order continuous norm, but is not a KB-space.

In Corollary~\ref{cor:main-everywhere-bl-kb} we required the underlying Banach lattice to be a KB-space.
Due to the irreducibility of $T$, we only need the space have order continuous norm in the following corollary.
In a Banach lattice with order continuous norm, every closed ideal is a projection band.%
\footnote{
	Indeed, every closed ideal in such a space is a band according to \cite[Theorem~II.5.14 on p.\,94]{Schaefer1974}; 
	moreover, a Banach lattice with order continuous norm is order complete \cite[Theorem~II.5.10 on p.\,89]{Schaefer1974}, 
	and in an order complete vector lattice every band is a projection band \cite[Theorem~II.2.10 on p.\,62]{Schaefer1974}.
}
Thus, instead of requiring the subspace $B$ in Theorem~\ref{thm:irred-bl} to be a projection band, 
one can require it to be a closed ideal in the following corollary.

\begin{corollary} 
	\label{cor:irred-bl-oc}
	Let $E$ be a complex Banach lattice with order continuous norm 
	and let $T: E \to E$ be a positive linear operator which is power bounded.
	Assume that $T$ is irreducible and that there exists a closed ideal $\{0\} \not= B \subseteq E$ 
	such that $\overline{E_{Tf}} \supseteq \overline{E_f}$ for each $0 \le f \in B$. 
	Then $\specPnt(T) \cap \bbT \subseteq \{1\}$.
\end{corollary}

\begin{proof}
	First note that, as observed before the corollary, $B$ is actually a projection band. 
	
	Now assume for a contradiction that $T$ has an eigenvalue $\lambda \in \bbT \setminus \{1\}$, 
	and let $z \in E$ be a corresponding eigenvector. 
	Then $T\modulus{z} \ge \modulus{Tz} = \modulus{z}$, 
	and thus it follows from the irreducibility and the power boundedness of $T$ 
	that $\modulus{z}$ is actually a fixed point of $T$, i.e., $T \modulus{z} = \modulus{z}$; 
	compare the proof of Theorem~\ref{thm:irred-bl}. 
	As also seen in the proof of Theorem~\ref{thm:irred-bl}, 
	the irreducibility of $T$ thus implies that $\modulus{z}$ is quasi-interior point of $E_+$.
	Since order intervals in Banach lattices with order continuous norm 
	are weakly compact \cite[Theorem~II.5.10 on p.\,89]{Schaefer1974}, 
	it thus follows by the same argument as in the proof of Corollary~\ref{cor:main-everywhere-bl-kb} 
	that $T$ is weakly almost periodic.
	
	Thus, Theorem~\ref{thm:irred-bl} is applicable, 
	and yields that $\specPnt(T) \cap \bbT \subseteq \{1\}$. 
	This is a contradiction.
\end{proof}

The second theorem in the introduction is an immediate consequence of the previous corollary:

\begin{proof}[Proof of Theorem~\ref{thm:main-irred}]
	Since $p \in [1,\infty)$, the Banach lattice $E := L^p$ has order continuous norm.
	Let $B \subseteq L^p$ denote the set of all functions which vanish outside $S$.
	This is a closed ideal in $L^p$, 
	and for each $0 \le f \in B$ the assumption $\supp(Tf) \supseteq \supp(f)$ 
	in Theorem~\ref{thm:main-irred}
	is equivalent to the inclusion $\overline{E_{Tf}} \supseteq \overline{E_{f}}$.
	Thus, the assumptions of Corollary~\ref{cor:irred-bl-oc} are satisfied,
	so the corollary shows that $\specPnt(T) \cap \bbT \subseteq \{1\}$.
\end{proof}

\subsection{Long-term behaviour}

Again, we use the spectral results from the previous subsections to derive consequences 
for the long-term behaviour of the powers of positive operators.

\begin{corollary}
	\label{cor:irred-bl-oc-convergence}
	Let $E$ be a complex Banach lattice with order continuous norm 
	and let $T: E \to E$ be a positive linear operator which is power bounded.
	Assume that $T$ is irreducible and that there exists a closed ideal $\{0\} \not= B \subseteq E$ 
	such that $\overline{E_{Tf}} \supseteq \overline{E_f}$ for each $0 \le f \in B$. 
	
	If $T$ has a non-zero fixed vector 
	and there exists an integer $n_0 \ge 0$ and a non-zero linear AM-compact operator $K: E \to K$ 
	such that $0 \le K \le T^{n_0}$, 
	then $T^k$ converges strongly as $k \to \infty$.
\end{corollary}

\begin{proof}
	According to Corollary~\ref{cor:irred-bl-oc} one has $\specPnt(T) \cap \bbT \subseteq \{1\}$.
	
	Due to the irreducibility of $T$ the assumptions~(a) and~(b) of 
	\cite[Theorem~3.9]{GerlachGlueck2019} are satisfied for the semigroup representation 
	$(T^k)_{k \in \bbN_0}$; 
	indeed, assumption~(a) follows from \cite[Proposition~3.11(c)]{GerlachGlueck2019}, 
	and assumption~(b) from the fact that every non-zero fixed point of $T$ in $E_+$ 
	is a quasi-interior point of $E_+$ 
	(see \cite[Theorem~V.5.2(i) on p.\,329]{Schaefer1974} or our proof of Theorem~\ref{thm:irred-bl}).
	
	If $x \in E$ denotes a non-zero fixed point of $T$, then $T\modulus{x} \ge \modulus{x}$, 
	so $\modulus{x}$ is even a fixed point of $T$ by property~(a) mentioned in the previous paragraph;
	therefore, $\modulus{x}$ is also a quasi-interior point according to property~(b). 
	Thus, the assumptions of \cite[Lemma~3.10]{GerlachGlueck2019} are satisfied
	(note that the divisibility assumption in \cite[Theorem~3.9]{GerlachGlueck2019} 
	is not needed for the lemma) 
	and this shows in turn that the standard assumptions of \cite[Section~6]{GlueckHaase2019}
	are satisfied. 
	As $T$ has no eigenvalues in $\bbT \setminus \{1\}$, 
	it thus follows from \cite[Theorem~6.1(c)]{GlueckHaase2019} 
	that $T^k$ converges strongly as $k \to \infty$.
\end{proof}

We do not discuss the case of integral operators on general measure spaces in detail here 
(which can be treated in the same spirit is 
in Corollaries~\ref{cor:everywhere-convergence-l-p} and~\ref{cor:lindeloef} in the preceding section), 
but we find it worthwhile to again mention the case of $\ell^p$-spaces explicitly:

\begin{corollary} 
	\label{cor:irred-l-p-convergence}
	Let $p \in [1,\infty)$ and set $\ell^p := \ell^p(\bbN)$. 
	Let $T: \ell^p \to \ell^p$ be a positive linear operator which is power bounded. 
	Assume that $T$ is irreducible and 
	has a fixed point $f_0$ such that $f_0(\omega) > 0$ for all $\omega \in \bbN$. 
	Assume moreover that $\langle e_j, Te_j \rangle > 0$ for at least one canonical unit vector $e_j$.

	Then $T^k$ converges strongly as $k \to \infty$.
\end{corollary}

\begin{proof}
	If $T$ is zero there is nothing to prove, so assume that $T \not= 0$. 
	For $p \in [1,\infty)$ every order interval in $\ell^p$ is compact, 
	so $T$ is AM-compact. 
	Hence, we can apply 
	Corollary~\ref{cor:irred-bl-oc-convergence}, 
	where $B := \big\{g \in \ell^p: \; g(\omega) = 0 \text{ for all } \omega \in \bbN \setminus \{j\}\big\}$.
\end{proof}

\section{Domination of the identity} 
\label{section:domination-of-id}

\subsection{The spectrum}

In this subsection we give to sufficient conditions for the peripheral spectrum of a positive operator 
to consist of the spectral radius only. 
The first of those results, Theorem~\ref{thm:peripheral-spectrum-first-power}, 
is likely to be known to experts in Perron--Frobenius theory. 
Yet, we could not find an explicit reference for it in the literature, so we include the proof.
By $\spr(T)$ we denote the spectral radius of a bounded linear operator $T$, 
and the \emph{peripheral spectrum of $T$} is denoted by
\begin{align*}
	\spec_\per(T) := \{\lambda \in \spec(T): \; \modulus{\lambda} = \spr(T)\}.
\end{align*}

\begin{theorem} 
	\label{thm:peripheral-spectrum-first-power}
	Let $E$ be a complex Banach lattice and let $T: E \to E$ be a positive linear operator. 
	If there is a number $\varepsilon > 0$ such that $T \ge \varepsilon \id$, 
	then $\spec_\per(T) = \{\spr(T)\}$.
\end{theorem}

\begin{proof}
	For a point $z$ in the complex plane and a number $r \ge 0$ we will use the notation
	$\overline{B}_r(z) \subseteq \bbC$ for the closed disk with center $z$ and radius $r$.
	
	Since $T$ is positive, we have $\spr(T) \in \spec(T)$ 
	\cite[Proposition~V.4.1 on p.\,323]{Schaefer1974}. 
	In particular, $\spr(T)$ is the largest real spectral value of $T$. 
	Since substracting $\varepsilon \id$ from the operator shifts the spectrum by $-\varepsilon$,
	it follows that $\spr(T)-\varepsilon$ is the largest real spectral value of $T-\varepsilon \id$.
	But $T-\varepsilon \id$ is also positive by assumption, 
	so $\spr(T-\varepsilon \id) \in \spec(T-\varepsilon \id)$
	(again by \cite[Proposition~V.4.1 on p.\,323]{Schaefer1974}), and
	hence it follows that $\spr(T-\varepsilon \id)$ is the largest real spectral value of $T-\varepsilon\id$.
	So we conclude that
	\begin{align*}
		\spr(T)-\varepsilon = \spr(T-\varepsilon \id).
	\end{align*}
	Hence, $\spr(T) - \varepsilon \ge 0$ and 
	$\spec(T-\varepsilon\id) \subseteq \overline{B}_{\spr(T)-\varepsilon}(0)$, and therefore,
	\begin{align*}
		\spec(T) 
		\subseteq 
		\overline{B}_{\spr(T)-\varepsilon}(0) + \varepsilon 
		= 
		\overline{B}_{\spr(T)-\varepsilon}(\varepsilon).
	\end{align*}
	Clearly, the latter disk intersects the circle $\spr(T) \bbT$ only in $\spr(T)$.
\end{proof}

\begin{remark}
	Theorem~\ref{thm:peripheral-spectrum-first-power} remains true if $E$ is not a Banach lattice, 
	but only (a complexification of) an order Banach space with generating and normal cone. 
	Indeed, on such a space one still knows that 
	the spectral radius of every positive operator is contained in its spectrum
	\cite[paragraph~2.2 on p.\,311]{SchaeferWolff1999},
	so the proof above works without modification.
\end{remark}

In our second theorem in this subsection we replace the assumption $T \ge \varepsilon \id$ 
with the more general assumption $T^n \ge \varepsilon T^{n-1}$ for an integer $n \ge 1$. 
In order to deduce from this condition that the peripheral spectrum is trivial, 
we need an additional technical assumption, namely that the peripheral spectrum is a priori known to be cyclic. 
We recall from the previous sections that a subset $S$ of the complex numbers if called \emph{cyclic} 
if $re^{i\theta} \in S$ (for some $r \ge 0$ and $\theta \in \bbR$) 
implies that $re^{in\theta} \in S$ for all $n \in \bbZ$
(see the discussion after the theorem for comments on this assumption).

\begin{theorem} 
	\label{thm:peripheral-spectrum-powers}
	Let $E$ be a complex Banach lattice, let $T: E \to E$ be a positive linear operator
	and assume that the peripheral spectrum of $T$ is cyclic. 
	If there is an integer $n \ge 1$ and a number $\varepsilon > 0$ such that $T^n \ge \varepsilon T^{n-1}$, 
	then $\spec_\per(T) = \{\spr(T)\}$.
\end{theorem}

The a priori assumption in the theorem that $\spec_\per(T)$ be cyclic is rather weak, 
since it is automatically satisfied for many classes of positive operators. 
In particular, if $\spr(T) = 1$ and $T$ is power bounded, 
which is a natural assumption if one wants to prove convergence of the iterates $T^n$, 
then positivity of $T$ always implies that $\spec_\per(T)$ is cyclic 
(this follows from \cite[Theorem~V.4.9 on p.\,327]{Schaefer1974}). 
It is, however, an open problem 
whether the peripheral spectrum of every positive operator on a Banach lattice is cyclic; 
see \cite{Glueck2018} and the references therein for a thorough discussion of this question.

\begin{proof}[Proof of Theorem~\ref{thm:peripheral-spectrum-powers}]
	There is no loss of generality in assuming that $\spr(T) = 1$.
	We have $0 \le T^n - \varepsilon T^{n-1} \le T^n$, 
	so $\spr(T^n - \varepsilon T^{n-1}) \le \spr(T^n) = \spr(T)^n = 1$. 
	
	Since $T$ is positive, one has $1 = \spr(T) \in \spec(T)$ 
	\cite[Proposition~V.4.1 on p.\,323]{Schaefer1974}. 
	Now assume for a contradiction that $\spec_\per(T)$ does not only consist of $\spr(T)$. 
	Since $\spec_\per(T)$ is cyclic, 
	this implies that there exists a number $\lambda \in \spec_\per(T)$ with real part $\re \lambda \le 0$. 
	For the number $\lambda^n - \varepsilon \lambda^{n-1}$, 
	which is contained in the spectrum of $T^n - \varepsilon T^{n-1}$, we then obtain
	\begin{align*}
		1 \ge 
		\spr(T^n - \varepsilon T^{n-1}) 
		\ge 
		\modulus{\lambda^n - \varepsilon\lambda^{n-1}} 
		= 
		\modulus{\lambda - \varepsilon} 
		>  
		\modulus{\lambda}
		=
		1,
	\end{align*}
	where the inequality at the end follows from $\re \lambda \le 0$.
	This is a contradiction.
\end{proof}

\subsection{Long-term behaviour}

Let us give an application of Theorem~\ref{thm:peripheral-spectrum-powers}, 
again to the asymptotics of the powers of $T$.
Recall that a bounded linear operator $T$ on a Banach space is called 
\emph{mean ergodic} if the sequence of its \emph{Cesàro means}
\begin{align*}
	\frac{1}{n} \sum_{k=0}^{n-1} T^k
\end{align*}
converges strongly as $n \to \infty$.
In this case, the limit of this sequence is called the \emph{mean ergodic projection} of $T$.

\begin{corollary}  
	\label{cor:peripheral-spectrum-powers-plus-ablp}
	Let $E$ be a complex Banach lattice and let $T: E \to E$ be a positive linear operator 
	which is power bounded and mean ergodic. 
	If there is an integer $n \ge 1$ and a number $\varepsilon > 0$ such that $T^n \ge \varepsilon T^{n-1}$, 
	then $T^k$ converges strongly as $k \to \infty$.
\end{corollary}

\begin{proof}
	The sequence of Cesàro means of $T$ is bounded in operator norm; 
	this follows from the mean ergodicity and the uniform boundedness principle. 
	Hence, the spectral radius of $T$ satisfies $\spr(T) \le 1$. 
	We may assume that $\spr(T) = 1$ since otherwise the assertion is trivial.
	
	Due to the boundedness of the Cesàro means, the resolvent of $T$ satisfies the estimate 
	\begin{align*}
		\sup_{\lambda \in (1,\infty)} (\lambda - 1) \norm{(\lambda - T)^{-1} } < \infty,
	\end{align*}
	see \cite[Theorem~1.7]{Emilion1985}. 
	This together with the positivity of $T$ implies that the peripheral spectrum of $T$ is cyclic
	\cite[Theorem~V.4.9 on p.\,327]{Schaefer1974}.
	
	So Theorem~\ref{thm:peripheral-spectrum-powers} is applicable 
	and yields that $\spec(T) \cap \bbT \subseteq \{1\}$.  
	Due to the mean ergodicity of $T$ we can apply the single operator version 
	of the so-called \emph{ABLV theorem} \cite[Theorem~5.1]{ArendtBatty1988}
	to obtain the assertion.%
	\footnote{More precisely speaking, we need to split $E$ into the range and the kernel of the mean ergodic projection $P$
	and apply the cited theorem to the restriction of $T$ to $\ker P$.}
\end{proof}

Let us close this section with an example which can loosely be interpreted 
as an infinite version of Example~\ref{exa:nagler}.
The interesting point about the example is that Corollary~\ref{cor:peripheral-spectrum-powers-plus-ablp} 
can be applied for the case $n = 2$, but not for $n=1$.

\begin{example}
	Endow the real line $\bbR$ with the Lebesgue measure and let $E := L^2(\bbR)$. 
	Let $(I_n)_{n \in \bbN}$ and $(J_n)_{n \in \bbN}$ be two partitions of $\bbR$ into measurable sets of measure $1$. 
	For each $f \in E$, the series
	\begin{align*}
		Tf := \sum_{n=1}^\infty \int_{I_n} f \dx x \cdot \one_{J_n}
	\end{align*}
	converges unconditionally in $E$, 
	and $T$ is a bounded linear operator on $E$ which has norm $1$ and is positive. 
	We now show the following two properties of $T$:
	\begin{enumerate}[label=(\alph*)]
		\item 
		The powers $T^k$ are not strongly convergent as $k \to \infty$, in general.
		
		\item 
		Assume now that there exists a number $\varepsilon > 0$ 
		such that $I_n \cap J_n$ has measure $\ge \varepsilon$ for each index $n \ge 1$. 
		Then $T^k$ converges strongly as $k \to \infty$.
	\end{enumerate}
	
	\begin{proof}
		(a) 
		If we choose $J_n = I_{n+1}$ for all $n \in \bbN$, 
		then we obtain $T^k \one_{I_1} = \one_{I_{k+1}}$ for all $k \ge 1$. 
		Since the sets $I_n$ are pairwise disjoint and have measure $1$, 
	 	it follows that the latter sequence is not convergent in $E$ as $k \to \infty$.
		
		(b) 
		Let $0 \le f \in E$. 
		Then we obtain
		\begin{align*}
			T^2f 
			= 
			\sum_{n=1}^\infty \int_{I_n} f \dx x \cdot T \one_{J_n} 
			\ge 
			\sum_{n=1}^\infty \int_{I_n} f \dx x \cdot \varepsilon \one_{J_n} 
			= 
			\varepsilon Tf,
		\end{align*}
		so $T^2 \ge \varepsilon T$. 
		Since $E$ is reflexive and $T$ is power bounded, 
		we know that $T$ is mean ergodic. 
		Therefore, Corollary~\ref{cor:peripheral-spectrum-powers-plus-ablp} 
		implies that $T^k$ is strongly convergent as $k \to \infty$.
	\end{proof}
\end{example}

\bibliographystyle{plain}
\bibliography{literature}

\begin{thebibliography}{10}

\bibitem{AliprantisBurkinshaw2006}
C.~D. Aliprantis and O.~Burkinshaw.
\newblock {\em Positive operators}.
\newblock Springer, Dordrecht, 2006.
\newblock Reprint of the 1985 original.

\bibitem{ArendtBatty1988}
W.~Arendt and C.~J.~K. Batty.
\newblock Tauberian theorems and stability of one-parameter semigroups.
\newblock {\em Trans. Amer. Math. Soc.}, 306(2):837--852, 1988.

\bibitem{BartoszekBrown1998}
Wojciech Bartoszek and Thomas~John Brown.
\newblock On {Frobenius}-{Perron} operators which overlap supports.
\newblock {\em Zesz. Nauk. Uniw. Jagiell., Univ. Iagell. Acta Math.},
  1223:183--186, 1998.

\bibitem{Davies2005}
E.~B. Davies.
\newblock Triviality of the peripheral point spectrum.
\newblock {\em J. Evol. Equ.}, 5(3):407--415, 2005.

\bibitem{Ding2003}
Yiming Ding.
\newblock The asymptotic behavior of {F}robenius-{P}erron operator with local
  lower-bound function.
\newblock {\em Chaos Solitons Fractals}, 18(2):311--319, 2003.

\bibitem{Emelyanov2004a}
Eduard Emel'yanov and Manfred Wolff.
\newblock Mean lower bounds for {M}arkov operators.
\newblock {\em Ann. Polon. Math.}, 83(1):11--19, 2004.

\bibitem{Emilion1985}
R.~{\'E}milion.
\newblock Mean-bounded operators and mean ergodic theorems.
\newblock {\em J. Funct. Anal.}, 61(1):1--14, 1985.

\bibitem{Engel2000}
K.-J. Engel and R.~Nagel.
\newblock {\em One-parameter semigroups for linear evolution equations}, volume
  194 of {\em Graduate Texts in Mathematics}.
\newblock Springer-Verlag, New York, 2000.
\newblock With contributions by S. Brendle, M. Campiti, T. Hahn, G. Metafune,
  G. Nickel, D. Pallara, C. Perazzoli, A. Rhandi, S. Romanelli and R.
  Schnaubelt.

\bibitem{Gerlach2013}
M.~Gerlach.
\newblock On the peripheral point spectrum and the asymptotic behavior of
  irreducible semigroups of {H}arris operators.
\newblock {\em Positivity}, 17(3):875--898, 2013.

\bibitem{GerlachGlueck2019}
Moritz Gerlach and Jochen Gl{\"u}ck.
\newblock Convergence of positive operator semigroups.
\newblock {\em Trans. Am. Math. Soc.}, 372(9):6603--6627, 2019.

\bibitem{GerlachGlueckKernel}
Moritz Gerlach and Jochen Gl{\"u}ck.
\newblock On characteristics of the range of integral operators.
\newblock 2022.
\newblock Preprint. Available at arxiv.org/abs/2205.14397v1.

\bibitem{Glueck2018}
Jochen Gl{\"u}ck.
\newblock Growth rates and the peripheral spectrum of positive operators.
\newblock {\em Houston J. Math.}, 44(3):847--872, 2018.

\bibitem{GlueckHaase2019}
Jochen Gl{\"u}ck and Markus Haase.
\newblock Asymptotics of operator semigroups via the semigroup at infinity.
\newblock In {\em Positivity and noncommutative analysis. Festschrift in honour
  of Ben de Pagter on the occasion of his 65th birthday. Based on the workshop
  ``Positivity and Noncommutative Analysis'', Delft, The Netherlands, September
  26--28, 2018}, pages 167--203. Cham: Birkh{\"a}user, 2019.

\bibitem{GlueckWeber2020}
Jochen Gl{\"u}ck and Martin~R. Weber.
\newblock Almost interior points in ordered {Banach} spaces and the long-term
  behaviour of strongly positive operator semigroups.
\newblock {\em Stud. Math.}, 254(3):237--263, 2020.

\bibitem{GlueckWolff2019}
Jochen Gl{\"u}ck and Manfred P.~H. Wolff.
\newblock Long-term analysis of positive operator semigroups via asymptotic
  domination.
\newblock {\em Positivity}, 23(5):1113--1146, 2019.

\bibitem{Grobler1995}
J.~J. Grobler.
\newblock Spectral theory in {B}anach lattices.
\newblock In {\em Operator theory in function spaces and {B}anach lattices},
  volume~75 of {\em Oper. Theory Adv. Appl.}, pages 133--172. Birkh\"auser,
  Basel, 1995.

\bibitem{Keicher2006}
V.~Keicher.
\newblock On the peripheral spectrum of bounded positive semigroups on atomic
  {B}anach lattices.
\newblock {\em Arch. Math. (Basel)}, 87(4):359--367, 2006.

\bibitem{KellerLenzVogtWojciechowski2015}
Matthias Keller, Daniel Lenz, Hendrik Vogt, and Rados{\l}aw Wojciechowski.
\newblock Note on basic features of large time behaviour of heat kernels.
\newblock {\em J. Reine Angew. Math.}, 708:73--95, 2015.

\bibitem{Krengel1985}
U.~Krengel.
\newblock {\em Ergodic theorems}, volume~6 of {\em de Gruyter Studies in
  Mathematics}.
\newblock Walter de Gruyter \& Co., Berlin, 1985.
\newblock With a supplement by Antoine Brunel.

\bibitem{Lasota1982}
A.~Lasota and James~A. Yorke.
\newblock Exact dynamical systems and the {F}robenius-{P}erron operator.
\newblock {\em Trans. Amer. Math. Soc.}, 273(1):375--384, 1982.

\bibitem{Lin2000}
Michael Lin.
\newblock Support overlapping {$L_1$} contractions and exact non-singular
  transformations.
\newblock {\em Colloq. Math.}, 84/85(part 2):515--520, 2000.
\newblock Dedicated to the memory of Anzelm Iwanik.

\bibitem{MakarowWeber2000}
B.~M. Makarow and M.~R. Weber.
\newblock On the asymptotic behavior of some positive semigroups.
\newblock 2000.
\newblock Preprint MATH-AN-09-2000, TU Dresden. Available from
  arxiv.org/abs/1901.04382v1.

\bibitem{Meyer-Nieberg1991}
P.~Meyer-Nieberg.
\newblock {\em Banach lattices}.
\newblock Universitext. Springer-Verlag, Berlin, 1991.

\bibitem{Nagler2015}
Johannes Nagler.
\newblock On the spectrum of positive linear operators with a partition of
  unity property.
\newblock {\em J. Math. Anal. Appl.}, 425(1):249--258, 2015.

\bibitem{Nagler2018}
Johannes Nagler.
\newblock Iterates of {Markov} operators and their limits.
\newblock {\em J. Math. Anal. Appl.}, 462(1):347--369, 2018.

\bibitem{Pichor2000}
Katarzyna Pich{\'o}r and Ryszard Rudnicki.
\newblock Continuous {M}arkov semigroups and stability of transport equations.
\newblock {\em J. Math. Anal. Appl.}, 249(2):668--685, 2000.

\bibitem{Rudnicki1995}
Ryszard Rudnicki.
\newblock On asymptotic stability and sweeping for {M}arkov operators.
\newblock {\em Bull. Polish Acad. Sci. Math.}, 43(3):245--262, 1995.

\bibitem{Schaefer1974}
H.~H. Schaefer.
\newblock {\em Banach lattices and positive operators}.
\newblock Springer-Verlag, New York, 1974.
\newblock Die Grundlehren der mathematischen Wissenschaften, Band 215.

\bibitem{SchaeferWolff1999}
H.~H. Schaefer and M.~P. Wolff.
\newblock {\em Topological vector spaces}, volume~3 of {\em Graduate Texts in
  Mathematics}.
\newblock Springer-Verlag, New York, second edition, 1999.

\bibitem{Wolff2008}
M.~P.~H. Wolff.
\newblock Triviality of the peripheral point spectrum of positive semigroups on
  atomic {B}anach lattices.
\newblock {\em Positivity}, 12(1):185--192, 2008.

\bibitem{Zaanen1983}
A.~C. Zaanen.
\newblock {\em Riesz spaces. {II}}, volume~30 of {\em North-Holland
  Mathematical Library}.
\newblock North-Holland Publishing Co., Amsterdam, 1983.

\bibitem{Zalewska-Mitura1994}
Anna Zalewska-Mitura.
\newblock A generalization of the lower bound function theorem for {M}arkov
  operators.
\newblock {\em Univ. Iagel. Acta Math.}, (31):79--85, 1994.

\end{thebibliography}

\end{document}